\documentclass[10pt,reqno]{amsart}
\usepackage[utf8x]{inputenc}
\usepackage[english]{babel}
\usepackage[T1]{fontenc}
\usepackage{amssymb,amsmath}
\usepackage[hidelinks]{hyperref}
\usepackage{url}
\usepackage{indentfirst}
\usepackage{enumerate,amsmath,amssymb, mathrsfs}

\usepackage{latexsym}
\usepackage{color}
\usepackage{accents}

\newcommand{\Acirc}{\accentset{\circ}{A}}
\allowdisplaybreaks

\newtheorem{prop}{Proposition}[section]
\newtheorem{thm}[prop]{Theorem}
\newtheorem{lem}[prop]{Lemma}
\newtheorem{coro}[prop]{Corollary}
\newtheorem*{rema}{Remark}

 \title[The area preserving Willmore flow]{The area preserving willmore flow and local maximizers of the Hawking mass in asymptotically Schwarzschild manifolds}
 \author{Thomas Koerber}
 \address{ Albert-Ludwigs-Universit\"at Freiburg,
 	Mathematisches Institut,
 	Eckerstr. 1
 	D-79104 Freiburg, Germany}
 \email{thomas.koerber@math.uni-freiburg.de}

\begin{document}

	\begin{abstract}
		We study the area preserving Willmore flow in an asymptotic region of an asymptotically flat manifold which is $C^3-$close to Schwarzschild. It was shown by Lamm, Metzger and Schulze that such an end is foliated by spheres of Willmore type, see \cite{lamm2011foliations}. In this paper, we prove that the leaves of this foliation are stable under small area preserving $W^{2,2}-$perturbations with respect to the area preserving Willmore flow. This implies, in particular, that the leaves are strict local area preserving maximizers of the Hawking mass with respect to the $W^{2,2}-$topology.
	\end{abstract}
	\maketitle
\section{Introduction}
Let $(M,g)$ be an asymptotically flat Riemannian three-manifold with non-negative scalar curvature. Under suitable decay conditions on the metric, such a manifold possesses a global non-negative invariant called the ADM mass and denoted by $m_{\operatorname{ADM}}$ (see \cite{arnowitt1961coordinate,schoen1979proof,bartnik1986mass}). On the other hand, finding the right notion of quasi-local mass corresponding to this global invariant remains an interesting open problem (see \cite{penrose1982some}). A promising candidate is the so-called Hawking mass $m_H$ defined by
$$
m_H(\Sigma):=\frac{|\Sigma|^{\frac12}}{(16\pi)^{\frac32}}\bigg(16\pi-\int_\Sigma H^2d\mu\bigg),
$$ 
where $\Sigma$ is a compact surface bounding a region $\Omega$ whose mass is to be determined. With the help of the Hawking mass, the ADM mass can be quantified in terms of the local geometry: in a celebrated work, Huisken and Illmanen used a weak version of the inverse mean curvature flow to prove the Riemannian Penrose inequality which states that the ADM mass of an asymptotically flat manifold is bounded from below by the Hawking mass of any connected outward minimizing surface (see \cite{huisken2001inverse}). A different version of the Penrose inequality, where the comparison surface is required to be minimal but not necessarily connected, was later on shown by Bray using a quasistatic flow (see \cite{bray2001proof}). More recently, Huisken introduced a concept of isoperimetric mass which only relies on $C^0$-data of the metric and provides a notion of quasi-local as well as global mass. The global mass can be shown to agree with the ADM mass in case the latter is well-defined. It turns out that the isoperimetric mass can be characterized in terms of the Hawking mass of outward minimizing surfaces (see \cite{huisken2006isoperimetric} or \cite{jauregui2016lower} for a more detailed discussion).
\\ \indent
While the Hawking mass enjoys such desirable connections to the global geometry, there are unfortunately many surfaces with negative Hawking mass. This is a contrast to some other concepts of quasi-local mass such as the Brown-York mass (see \cite{shi2002positive}).  It was therefore a crucial insight by Christodolou and Yau that  the Hawking mass of a closed stable constant mean curvature surface is non-negative (see \cite{christodoulou1988some}). This suggested that such surfaces are suitable to test the gravitational field of an asymptotically flat manifold and motivated the study of the isoperimetric problem in such spaces. As some of the following results require stronger decay conditions on the metric than asymptotical flatness, we make the following definition: the metric $g$ is said to be $C^k-$close to Schwarzschild with decay coefficient $\eta$ and ADM mass $m$ if in the chart at infinity there holds $g=g_S+h$, where $h$ a symmetric two tensor satisfying 
\begin{align}
|\partial^j h|\leq \eta r^{-2-j},\ \label{decay equation}
\end{align}
for any $0\leq j\leq k$.
Here, $g_S$ is the Schwarzschild metric with mass $m$, $\partial$ the Euclidean derivative and $r$ the radial parameter in the chart at infinity. The Schwarzschild space models a static, single black hole and a space which is $C^k-$close can be understood to be a small perturbation. The first breakthrough in the study of the isoperimetric problem was accomplished by
Huisken and Yau, who used a volume-preserving version of the mean curvature flow to show that an asymptotic region of an asymptotically flat manifold which is $C^4-$close to Schwarzschild and has non-negative scalar curvature is foliated by embedded stable constant mean curvature spheres. Such a foliation induces a natural coordinate system and also gives rise to a geometric center of mass. This result was later on refined by Qing and Tian who showed uniqueness of this foliation (see \cite{qing2007uniqueness}).
Using an ingenious argument, Bray showed in his PhD-thesis that the centred spheres are the unique non-null-homologous isoperimetric surfaces in the Schwarzschild manifold (see \cite{bray2009penrose}). This provided evidence that the leaves of the foliation in \cite{huisken1996definition} might actually be isoperimetric. In another breakthrough, Eichmair and Metzger extended the idea of Bray and showed in \cite{eichmair2013unique} that a foliation as in \cite{huisken1996definition} exists even if the manifold is only $C^2-$close to Schwarzschild. Furthermore, they proved that in the asymptotic region the leaves are in fact the unique isoperimetric surfaces enclosing a sufficiently large volume. It was later on shown by Chodosh, Eichmair, Shi and Yu that a unique minimizer of the isoperimetric problem exists even if the manifold is only asymptotically flat and satisfies a certain decay condition on the scalar curvature (see \cite{chodosh2016isoperimetry}). On the other hand, studying the uniqueness of stable constant mean curvature spheres which are not necessarily isoperimetric
turned out to be a more difficult problem. As a first step in this direction,
Brendle showed a Heintze-Karcher type inequality and used a conformal flow in an elegant way to show that the centred spheres are the only constant mean curvature surfaces contained in one half of the Schwarzschild manifold  (see \cite{brendle2013constant}).  Finally, in a series of crucial results, Chodosh and Eichmair obtained the unconditional characterization of stable constant mean curvature surfaces in asymptotically flat manifolds which are $C^6-$close to Schwarzschild and whose scalar curvature is non-negative and satisfies a certain decay condition. By comparing certain mass flux integrals (\cite{chodosh2017global}) and using a  Lyapunov-Schmidt analysis to study null-homologous surfaces (\cite{chodosh2017far}), they showed that the leaves of the foliation are the only stable compact constant mean curvature surfaces without any assumption on their homology class. In the proof, the result \cite{carlotto2016effective} by Carlotto, Chodosh and Eichmair played an important part where they showed among other things that any asymptotically flat manifold with non-negative scalar curvature admitting an unbounded area minimizing minimal surface must be isometric to the flat Euclidean space. The results in \cite{chodosh2017global} seem to be optimal in some sense (see also \cite{brendle2014large}). Moreover, it should be noted that they stand  in stark contrast to the situation in the Euclidean space. The presence of positive mass seems to rule out all but one isoperimetric surface. 
\\\indent
For any concept of quasi-local mass, it is natural to look for regions which contain a maximal amount of mass. Usually, one can only hope to find such regions if one fixes a certain geometric quantity such as the volume of the region $\Omega$ or the area of its boundary $\Sigma$. In the case of the Hawking mass, the latter seems to be the more natural quantity.
While isoperimetric surfaces enjoy non-negative Hawking mass, they in fact maximize Huisken's quasi-local isoperimetric mass when fixing the volume of $\Omega$.  Hence, when studying the Hawking mass, it might be a more natural problem to directly look  for maximizers of the Hawking mass when fixing the area.
This approach is equivalent to finding area-constrained minimizers  of the Willmore functional $\mathcal{W}$ which is defined to be
\begin{align*}
\mathcal{W}(\Sigma):=\frac14\int_\Sigma H^2 d\mu.
\end{align*}
While the isoperimetric problem can be formulated solely in terms of $C^0$-data, the Hawking mass depends on higher order quantities and therefore seems to be more complicated to investigate. In fact, the Euler-Lagrange equation for the Willmore functional is a fourth-order elliptic equation and cannot be studied with the same techniques as the second order problem of finding constant mean curvature surfaces. Nevertheless,  using a continuity method and integral curvature estimates, Lamm, Metzger and Schulze showed the following result (see \cite{lamm2011foliations} and section 2 for a more precise statement).
\begin{thm} Let $(M,g)$ be an asymptotically flat manifold which is $C^3-$close to Schwarzschild, with mass $m>0$ and decay coefficient $0<\eta\leq\eta_0$ for some constant $\eta_0$ depending only on $m$ and satisfies $|\operatorname{Sc}|\leq \eta r^{-5}$. Then there exists a constant $\lambda_0>0$ and a compact set $K$ depending only on $m$ and $\eta_0$ such that $M\setminus K$ is foliated by surfaces of Willmore type $\Sigma_\lambda$ where $\lambda \in(0,\lambda_0)$. Moreover, every sufficiently centered, strictly mean-convex sphere $\Sigma\subset M\setminus K$ which is of Willmore type belongs to this foliation. 
	\label{lmsthm rough}
\end{thm}
Here, a surface of Willmore type is a critical point of the area prescribed Willmore energy. More precisely, every $\Sigma_\lambda$ satisfies the equation
$$
\Delta H+H(\operatorname{Rc}(\nu,\nu)+|\Acirc|^2+\lambda)=0.
$$
As for the isoperimetric problem, the positivity of the ADM mass is related to uniqueness which is evidently violated in the Euclidean space. The leaves of the foliation enjoy various desirable properties: if the scalar curvature is non-negative, the Hawking mass is positive and non-decreasing along the foliation and approaches the ADM mass as $\lambda\to 0$.  Given the results obtained for the isoperimetric problem, one is tempted to believe that in an asymptotic region, the leaves $\Sigma_\lambda$ are the global maximizers of the Hawking mass and perhaps the only surfaces of Willmore type with non-negative Hawking mass and a sufficiently large area. Up to now, this has not even been known in Schwarzschild. In fact, a result comparable to the one obtained in \cite{bray2001proof} cannot be expected as one can easily construct spheres which are close and homologous to the  horizon, but have arbitrarily large Hawking mass. On the other hand, it is possible to construct off-center surfaces whose Hawking mass is arbitrarily close to the ADM-mass which in turn equals the Hawking mass of the centred spheres (see the remark below Corollary \ref{main coro}). Hence, the problem of maximizing the Hawking mass even with fixed area seems particularly challenging from a variational point of view. \\ In this work we make partial progress in understanding the role of the centred spheres in the Schwarzschild space or more generally of the leaves $\Sigma_\lambda$ in the foliation constructed in \cite{lamm2011foliations} in asymptotically Schwarzschild spaces regarding the maximization of the Hawking mass. In fact, we show the following:
\begin{thm}
	Let $(M,g)$ be $C^3-$close to Schwarzschild with decay coefficient $0<\eta\leq\eta_0(m)$ and mass $m>0$ and let $\{\Sigma_\lambda | \lambda\in(0,\lambda_0)\}$ be the foliation from the previous theorem. Then there is a constant $\Lambda<\lambda_0$ depending only on $m$ and $\eta_0$ such that any immersed sphere $\Sigma\subset M$ enclosing $\Sigma_{\Lambda}$ which is sufficiently centred satisfies $m_H(\Sigma)\leq m_H(\Sigma_\lambda)$, where $\lambda$ is chosen such that $|\Sigma|=|\Sigma_\lambda|$. If equality holds, then $\Sigma=\Sigma_\lambda$. Moreover, if the scalar curvature of $(M,g)$ is non-negative, there holds $m_H(\Sigma)\leq m$ with equality if and only if $M$ is a centred sphere in the spatial Schwarzschild manifold. 
	\label{main theorem1}
\end{thm} 
A more local version of this result in the deSitter-Schwarzschild space was shown by Maximo and Nunes, see \cite{maximo2012hawking}. They considered graphical surfaces with respect to the centred spheres and computed the second variation of the Hawking mass. Our approach relies instead on a stability result for the area preserving Willmore flow which we will discuss below. 
\begin{thm}
	Let $(M,g)$ be $C^3-$close to Schwarzschild with decay coefficient $0<\eta\leq\eta_0(m)$ and mass $m>0$ and let $\Sigma$ be an embedded sphere which is obtained as a small area-preserving $W^{2,2}-$perturbation of a leave of the foliation $\{\Sigma_\lambda\}$ from Theorem \ref{lmsthm rough}. Then the area preserving Willmore  flow starting at $\Sigma$ exists for all times and converges smoothly to one of the leaves in the foliation $\{\Sigma_\lambda\}$. 
	\label{main theorem2}
\end{thm}
For a more precise statement of Theorem \ref{main theorem1} and Theorem \ref{main theorem2} we refer to Theorem \ref{main result} and Corollary \ref{main coro}.
The Willmore flow was introduced by Kuwert and Sch\"atzle in \cite{kuwert2001willmore,kuwert2002gradient} as the $L^2-$gradient flow for the Willmore energy in the Euclidean space and it has been studied in various contexts ever since. The area preserving Willmore flow is the $L^2-$projection of the Willmore flow onto area preserving immersions and was introduced by Jachan in his PhD-thesis, see \cite{jachan2014area}. It is a smooth one-parameter family of surfaces leaving the area constant while decreasing the Willmore energy and hence increases the Hawking mass. Using methods similar to Kuwert and Schätlze in \cite{kuwert2001willmore,kuwert2002gradient,kuwert2004removability}, Jachan showed long time existence and subsequential convergence for topological spheres to a surface of Willmore type requiring a specific bound on the initial energy  and assuming that the flow avoids a sufficiently large compact set for all times. One might therefore expect that the area-preserving Willmore flow can be used to produce area-prescribed critical points of the Hawking mass. It is however in general not clear under which initial conditions this assumption can be verified. \\In order to prove Theorem \ref{main theorem2}, we verify the constraints of the long time existence result of \cite{jachan2014area}. By a result of Müller and deLellis, surfaces with small traceless part of the second fundamental form are $W^{2,2}-$close to a round sphere. From this it follows that proving long time existence eventually reduces to controlling the evolution of the barycentre. In order to do this, we derive a differential inequality which we calculate in terms of the approximating round sphere. It turns out that this inequality is governed by the positivity of mass. The crucial part is then to control the error terms which requires precise estimates of the evolving geometric quantities. Then, this argument can be used to show that the barycentre does not move too much if it is initially not too far away from the origin.  Finally, we can use the uniqueness statement in Theorem \ref{lmsthm rough}.  to identify the limit of the flow and deduce smooth convergence. Theorem \ref{main theorem1} then follows from Theorem \ref{main theorem2} and the fact that sufficiently centred surfaces can either be flown back to a leave of the foliation or have negative Hawking mass.
\\
The rest of this paper is organized as follows. In section 2, we fix some notation, collect results about asymptotically flat manifolds which are $C^3-$close to Schwarzschild and show that the area preserving Willmore flow of a small $W^{2,2}-$perturbation of a leave of the canonical foliation remains round and avoids a large compact set if its barycenter does not move too much.
In section 3, we prove general integral curvature estimates in the spirit of \cite{kuwert2001willmore} for asymptotically Schwarzschild manifolds.  In section 4, we combine these estimates with a careful analysis of the evolution equation for the Willmore energy to obtain precise a-priori estimates for certain geometric quantities under the area preserving Willmore flow. Finally, in section 5, we  derive a differential inequality for the barycentre to find that the evolution is governed by the translation sensitivity of the Willmore energy in the Schwarzschild space. We then proceed to  prove Theorem \ref{main theorem2}. For the convenience of the reader, we have included a summary of the argument used by Jachan in \cite{jachan2014area} in the appendix.
\\ \textbf{Acknowledgements.} The author would like to thank Felix Schulze for suggesting the problem and for many helpful discussions. This research was carried out at the Department of Mathematics of  University College London and the author would like to thank  for the hospitality. Finally, the author gratefully acknowledges financial support from the DAAD.

\section{Preliminaries}
Let $(M,g)$ be an a three dimensional, complete and asymptotically flat Riemannian manifold which is $C^3-$close to Schwarzschild with mass $m>0$ and decay coefficient $\eta>0$. More precisely, we assume that there is a compact set $K$ such that $M\setminus K$ is diffeomorphic to $R^{3}\setminus\overline{B_{\sigma}(0)}$ for some $\sigma>m/2$ and that the following estimate holds on $M\setminus K$:
\begin{align}
r^2|g-g_S|+r^3|\overline\nabla_g-\overline\nabla_S|+r^4|\operatorname{Rc}_g-\operatorname{Rc}_S|+r^5|\overline{\nabla}_g\operatorname{Rc}_g-\overline{\nabla}_S\operatorname{Rc}_S|\leq \eta.
\label{asymptotic behaviour}
\end{align}
Here, $r$ denotes the radial function of the asymptotic chart $\mathbb{R}^3\setminus \overline{B_{\sigma}(0)}$, $\overline\nabla_g$ the gradient of the ambient space and $\operatorname{Rm}_g$ and $\operatorname{Rc}_g$ the Riemann curvature tensor and the Ricci curvature of the ambient space, respectively. $g_S$ denotes the Schwarzschild metric with mass $m$, which is defined by
$$
g_S:=\bigg(1+\frac{m}{2r}\bigg)^4g_e=\phi^4 g_e.
$$
The subscripts $S$ and $e$ indicate that the geometric quantity is computed with respect to the Schwarzschild and the Euclidean metric, respectively. On the other hand, we will usually omit the subscript $g$.  The definition of being $C^{3}-$close to Schwarzschild is equivalent to the definition given in the introduction and the mass parameter $m$ is of course equal to the ADM-mass of $(M,g)$. 
Finally, we assume that the scalar curvature of $(M,g)$, denoted by $\operatorname{Sc}_g$, satisfies
\begin{align}
|r^5\operatorname{Sc}_g|\leq \eta.
\end{align}
Such manifolds were called asymptotically Schwarzschild in \cite{lamm2011foliations} and we will adopt this terminology from now on. The Schwarzschild manifold $(M_S,g_S)=(\mathbb{R}^3\setminus\{0\},g_S)$ is the model space for the problem studied in this paper. It models a single static black hole with mass $m>0$ and the metric being static is expressed in the following equation
\begin{align}
(\overline\Delta_S f) g_S -\overline\nabla^2_Sf+f\operatorname{Rc}_S= -\overline\nabla^2_Sf+f\operatorname{Rc}_S=0, \label{static equation}
\end{align}
where $f=(2-\phi)/\phi$ is the so-called potential function. It follows that  $(M_S,g_S)$ has vanishing scalar curvature and that the Ricci curvature is given by
\begin{align}
\operatorname{Rc}_S(\cdot,\cdot)=mr^{-3}\phi^{-2}(g_e(\cdot,\cdot)-3g_e(\partial_r,\cdot)g_e(\partial_r,\cdot)). \label{rc schwarzschild}
\end{align}
We consider an immersed, closed and orientable surface $\Sigma\subset M$ and denote its first fundamental form by $\gamma$, its connection by $\nabla$, its outward normal by $\nu$, its second fundamental form by $A$, the traceless part by $\Acirc$, the mean curvature by $H$ and the area element by $d\mu$. Moreover, we denote the induced curvature by $\operatorname{Rc}^\Sigma$ and $\operatorname{Rm}^\Sigma$, respectively. $\Sigma$ can also be regarded as an embedded surface in $(\mathbb{R}^3\setminus\overline{B_{\sigma}(0)},g_e)$ or $(\mathbb{R}^3\setminus\overline{B_{\sigma}(0)},g_S)$. We indicate the corresponding geometric quantities by the subscripts $e$ and $S$. If we want to emphasize the correspondence to $g$, we sometimes use the subscript $g$.
We use the letter $c$ for any constant that only depends on $m,\eta$ in a non-decreasing way. The meaning of such a constant will be different in most of the following inequalities. If a constant has a geometric dependency, we will explicitly state it.  We fix a chart at infinity and extend the radial parameter $r$ in a smooth way to all of $M$. We define $r_{\min}$ and $r_{\max}$ to be the minimal and maximal value of the radial function on $\Sigma$, respectively. As all of our results concern surfaces which are contained in the asymptotic region, we will always assume that $r_{\min}\geq R_0$ for some positive constant $R_0(\eta,m)$ which is to be determined. We let $x$ be the position vector in the asymptotic region. If $\Sigma$ is contained in the asymptotic region, we define the Euclidean barycentre, approximate radius and centring parameter by
$$
a_e=\frac{\int_\Sigma xd\mu_e }{|\Sigma|_g}, \qquad R_e:=\frac{|\Sigma|^{\frac12}_e}{\sqrt{4\pi}}, \qquad \tau_e=\frac{|a_e|}{R_e},
$$
respectively, and reiterate these definitions for the metrics $g$ and $g_S$\footnote{We still use the Euclidean norm in the definition of the centring parameter with respect to the metrics $g$ and $g_S$.}.  The Hawking mass $m_H$ of a surface $\Sigma$ is defined by
\begin{align}
m_H(\Sigma)=\frac{|\Sigma|^{1/2}}{(16\pi)^{3/2}}\bigg(16\pi-\int_\Sigma H^2d\mu \bigg)=\frac{|\Sigma|^{1/2}}{(16\pi)^{3/2}}\bigg(16\pi-4\mathcal{W}(\Sigma)\bigg),
\label{hawking mass def}
\end{align}
where $\mathcal{W}$ denotes the Willmore energy, that is,
$$
\mathcal{W}(\Sigma)=\frac14\int_\Sigma H^2d\mu.
$$
In \cite{lamm2011foliations}, Lamm, Metzger and Schulze studied area-prescribed critical points of the Willmore energy and called them surfaces of Willmore type. Such surfaces satisfy the equation
\begin{align}
\Delta H+H(\operatorname{Rc}(\nu,\nu)+|\Acirc|^2+\lambda)=0 \label{willmore type}
\end{align}
for some scalar parameter $\lambda$. As Lamm, Metzger and Schulze  showed by using a continuity method, an asymptotic region of an asymptotically Schwarzschild manifold is foliated by such surfaces. More precisely, they showed the following theorem.
\begin{thm} Let $(M,g)$ be  an asymptotically Schwarzschild manifold with mass $m$ and decay coefficient $\eta\leq\eta_0$ for some $\eta_0>0$ depending only on $m$. Then there exists a constant $\lambda_0>0$ and a compact set $K$ depending only on $m,\eta_0$ such that $M\setminus K$ is foliated by embedded spheres of Willmore type $\Sigma_\lambda$ where $\lambda \in(0,\lambda_0)$. Moreover, there are constants $\chi, \tilde \tau>0$ which only depend on $m,\eta_0$ such that any strictly mean convex sphere $\Sigma\subset M\setminus K$ of Willmore type satisfying $\tau_e\leq \tilde\tau$ and $R_e\leq \chi r_{\min}^2$ belongs to the foliation. 
	\label{lmstheorem}
\end{thm}
 Another tool to find surfaces of Willmore type is the so-called area preserving Willmore flow which was introduced by Jachan in \cite{jachan2014area}. In this paper, we will study the evolution of spherical surfaces under the area preserving Willmore flow which is defined as follows. We say that a smooth family of surfaces $\{\Sigma_t|0\leq t\leq T\}$, $\Sigma_0=\Sigma$ flows by the area preserving Willmore flow with initial data $\Sigma$ if it satisfies the following evolution equation:
\begin{align}
\label{flow equation}
\frac{d}{dt}x=(\Delta H + H(\operatorname{Rc}(\nu,\nu)+|\Acirc|^2+\lambda))\nu=:W\nu+\lambda H\nu,
\end{align}
where the Lagrange parameter $\lambda$ is given by
\begin{align}
\lambda(t)=|H|^{-2}_{L^2(\Sigma_t)}\int_{\Sigma_t}(|\nabla H|^2-H^2\operatorname{Rc}(\nu,\nu)-H^2|\Acirc|^2)d\mu.
\label{lagrange}
\end{align}
The area preserving Willmore flow is the $L^2$-projection of the Willmore flow onto the class of area-preserving flows. It is easy to see that this evolution leaves the area constant, decreases the Willmore energy and consequently increases the Hawking mass. In fact, for a flow with normal speed $f$ there holds
\begin{align}
\frac{d}{dt}H=-\Delta f-f|A|^2-f\operatorname{Rc}(\nu,\nu),\qquad \frac{d}{dt}d\mu=fHd\mu.
\label{ev equations}
\end{align}
On the other hand, integration by parts reveals that
\begin{align}
\lambda(t)\int_{\Sigma}H^2d\mu=\int_{\Sigma} WHd\mu. \label{lambda in terms of W}
\end{align}
Plugging in $f=W+\lambda H$, integrating by parts and using the identity $2|A|^2=H^2+2|\Acirc|^2$ we find
\begin{align}
\frac{d}{dt}\int_{\Sigma} H^2d\mu&=\int_{\Sigma}\bigg(2H(-\Delta(W+H\lambda)-|A|^2(W+H\lambda)-\operatorname{Rc}(\nu,\nu)(W+H\lambda))+H^3(W+\lambda H)\bigg)d\mu \notag \\ \notag
&=\int_{\Sigma} \bigg(2W(-\Delta H-|\Acirc|^2-\operatorname{Rc}(\nu,\nu))+\lambda(|\nabla H|^2-|\Acirc|^2-\operatorname{Rc}(\nu,\nu))\bigg)d\mu\\
&=-2\int_{\Sigma} W^2d\mu+2\bigg(\int_{\Sigma} H^2\bigg)^{-1}d\mu\bigg(\int_\Sigma WHd\mu\bigg)^2 \label{h2 growth}
\\&\leq 0, \notag
\end{align}
where we used Hölder's inequality in the last step. Similarly,
\begin{align}
\frac{d}{dt}|\Sigma|=\int_{\Sigma} WHd\mu+\lambda(t)\int_{\Sigma}H^2d\mu=0,\label{area const}
\end{align}
where we used (\ref{lambda in terms of W}). In \cite[Theorem 1.6, Theorem 1.7]{jachan2014area}, Jachan showed the following:
\begin{thm}
	Let $\Sigma$ be an immersed surface homologous to a coordinate sphere. Then a solution to the equation (\ref{flow equation}) exists on a maximal time interval $[0,T)$.  Moreover, if $\mathcal{W}(\Sigma)\leq 8\pi -\rho$ for some constant $\rho>0$,  there exists $R_0=R_0(m,\eta,\rho^{-1})$ such that if $r_{\min}(t)\geq R_0$ for all $t\in[0,T)$, then $T=\infty$. If this is true,  flow subsequentially converges smoothly to a surface of Willmore type.
	\label{jachan existence}
\end{thm}
For the convenience of the reader, we have included an outline of the proof in the appendix. We also remark that a similar problem has been studied by Link in \cite{link2013gradient}. From (\ref{h2 growth}) and (\ref{area const}) it follows that the Hawking mass is non-decreasing along the area preserving Willmore flow. It is consequently a natural flow to find local maximizers of the Hawking mass. In this paper, we will study the stability of the foliation constructed in \cite{lamm2011foliations} under small $W^{2,2}-$perturbations with respect to the area preserving Willmore flow. To this end, we say that a surface $\Sigma$ is \textit{admissible} if it is an embedded sphere satisfying the following conditions:
  \begin{align}
 \label{hypothesis a}
 |\Acirc|_{L^2(\Sigma)}^2 \leq \epsilon, \qquad \tau_e \leq \delta,  \qquad r_{\min} \geq  R_0,
 \end{align}
for some constants $\epsilon, \delta>0$. Moreover, we say that an area preserving Willmore flow is \textit{admissible} up to time $T>0$ if every $\Sigma_t$ with $0\leq t\leq T$ is admissible. Before we proceed, we need the following lemma to relate geometric quantities with regards to the different background metrics. The lemma is a straight-forward consequence of the asymptotic behaviour of the metric (\ref{asymptotic behaviour}), c.f. \cite[section 1]{lamm2011foliations}.
\begin{lem} Let $\Sigma$ be an embedded sphere satisfying $r_{\min}\geq R_0$ for some $R_0(m,\eta)$ sufficiently large.
	There is a universal constant $c$ such that
	\begin{align*}
	\nu_S=\phi^{-2}\nu_e, \qquad&
	|\nu_S-\nu|\leq c\eta r^{-2}, \\
	\nabla_S=\nabla_e+\phi^{-1}\nabla_e\phi, \qquad&
	|\nabla_S-\nabla|\leq c\eta r^{-3}, \\
	H_S=\phi^{-2}H_e-2mr^{-2}\phi^{-3}\nu_e\cdot \partial_r, \qquad
	&|H_S-H|\leq c\eta r^{-3}+c\eta r^{-2}|A| , \\
	\Acirc_S=\phi^{2}\Acirc_e, \qquad& |\Acirc_S-\Acirc|\leq c\eta r^{-3}+c\eta r^{-2}|A|, \\
	|A_S-A_e|\leq cmr^{-2}+cmr^{-1}A_S, \qquad& |A_S-A|\leq c\eta r^{-3}+c\eta r^{-2}|A|, \\
	d\mu_S=\phi^4 d\mu_e, \qquad& |d\mu_S -d\mu|\leq c\eta r^{-2}.
	\end{align*}
	as well as
	\begin{align*}
	|\nabla_SA_S-\nabla_e A_e|&\leq cmr^{-3}+cmr^{-2}|A_S|+cmr^{-1}|\nabla_S A_S|, \\
	|\nabla_SA_S-\nabla A|&\leq c\eta r^{-4}+c\eta r^{-3}|A|+c\eta r^{-2}|\nabla A|, \\|\nabla_SH_S-\nabla_e H_e|&\leq cmr^{-3}+cmr^{-2}|A_S|+cmr^{-1}|\nabla_S A_S|, 
	\\|\nabla_SH_S-\nabla H|&\leq c\eta r^{-4}+c\eta r^{-3}|A|+c\eta r^{-2}|\nabla A|.
	\end{align*}
	By conformal invariance, there holds 
	\begin{align*}
	|\Acirc_S|_{L^2(\Sigma)}&=|\Acirc_e|_{L^2(\Sigma)}, \\
	||H_S|^2_{L^2(\Sigma)}-|H_e|^2_{L^2(\Sigma)}|&\leq cmr_{\min}^{-1}|H_S|^2_{L^2(\Sigma)}.
	\end{align*} Finally,  there holds
	\begin{align*}
	|\Acirc|^2_{L^2(\Sigma)}&\leq (1+c\eta r^{-2}_{\min})|\Acirc_e|^2_{L^2(\Sigma)}+c\eta r_{\min}^{-3} |H|^2_{L^2(\Sigma)},
	\\||H|^2_{L^2(\Sigma)}-|H_S|^2_{L^2(\Sigma)}|&\leq c\eta r_{\min}^{-2}+ c\eta r_{\min}^{-2} |H|_{L^2(\Sigma)}^2. 
	\end{align*}
	\label{identities}
\end{lem}
Admissible surfaces enjoy various properties:
Müller and de Lellis showed (see \cite{de2005optimal,de2006c}) that in the Euclidean space, a surface $\Sigma$ with small traceless part of the second fundamental form in the $L^2-$sense is $W^{2,2}-$close to a round sphere $S$ and that there is a conformal parametrization mapping $S$ onto $\Sigma$. Using the conformal parametrization, geometric quantities on $S$ and $\Sigma$ can be related. The corresponding quantities on $S$ will be indicated by a tilde. The exact statement of the result in \cite{de2005optimal,de2006c} is as follows.
\begin{lem}
	\label{muldelel}
	Let $\Sigma\subset\mathbb{R}^3$ be a surface satisfying $|\Acirc_e|^2_{L^2(\Sigma)}<8\pi$. Then $\Sigma$ is a sphere and there exists a universal constant $c$ independent of $\Sigma$ and a conformal parametrization $\psi:S:=S^2_{a_e}(R_e)\to\mathbb{R}^3$ such that
	\begin{align*}
	|\psi-\operatorname{id}|_{L^2(S)}&\leq c	R_e^2 |\Acirc_e|_{L^2(\Sigma)}, \\
	|\nabla_e \psi-\tilde \nabla_e\operatorname{id}|_{L^2(S)}&\leq c	R_e |\Acirc_e|_{L^2(\Sigma)}, \\
	|\nabla_e ^2 \psi-\tilde \nabla_e ^2 \operatorname{id}|_{L^2(S)}&\leq c |\Acirc_e|_{L^2(\Sigma)}, \\
	|\nu_e\circ \psi - \tilde \nu_e\circ \operatorname{id}|_{L^2(S)}& \leq c R_e |\Acirc_e|_{L^2(\Sigma)}, \\
	|\nabla_e \nu_e\circ \psi - \tilde\nabla_e \tilde \nu_e\circ \operatorname{id}|_{L^2(S)}& \leq c |\Acirc_e|_{L^2(\Sigma)}, \\
	|h^2-1|_{L^\infty(S)} &\leq c |\Acirc_e|_{L^2(\Sigma)},
	\end{align*}
	where $\operatorname{id}:S\to\mathbb{R}^3$ is the identity, $h$ the conformal factor of $\psi$ and $\nu_e,\tilde \nu_e$ the outward normals of $\psi$ and $\operatorname{id}$, respectively. Moreover, the Sobolev embedding theorem implies
	$$
	|\psi-\operatorname{id}|_{L^\infty(S)}\leq c	R_e |\Acirc_e|_{L^2(\Sigma)} \\
	$$ 
	and one easily obtains
	$$
	|\tilde A_e-A_e|_{L^2(S)}\leq c |\Acirc_e|_{L^2(\Sigma)}.
	$$
\end{lem}
Using the previous two lemmas, we can relate $r_{\min},r_{\max},R_e$ provided $\epsilon,\delta$ are chosen sufficiently small in (\ref{hypothesis a}):
\begin{lem} Let $|\Acirc_e|^2_{L^2}<8\pi$. There holds
	$$
	(1-\tau_e-c|\Acirc_e|_{L^2})R_e\leq r_{\min} \leq r_{\max} \leq (1+\tau_e+c|\Acirc_e|_{L^2})R_e.
	$$
	There are constants $\kappa,\tau_0>0$ such that  
	$$c^{-1}R_e\leq r_{\min} \leq r_{\max} \leq cR_e,
	$$
	provided $|\Acirc_e|_{L^2}<\kappa$ and $\tau_e \leq \tau_0$. In particular, $\epsilon,\delta$ in (\ref{hypothesis a}) can be chosen such that  the previous estimate is true for any admissible surface.
	\label{r compare}
\end{lem}
We also need to compare the quantities $|\Sigma|,R,\tau,a$ with respect to the different background metrics.

\begin{lem}
	Let $\Sigma$ be an admissible surface. Then we have the following estimates
	\begin{align*}
	|\Sigma|_g c^{-1}\leq &|\Sigma|_e  \leq |\Sigma|_g c,
	\\
	|R_g| c^{-1}\leq &|R_e| \leq |R_g| c,
	\\
	|R_g-R_S|\leq c\eta R_e^{-1}, & \qquad |R_S-R_e|\leq cm,
	\\
	|a_g-a_e|\leq c\eta R_e^{-1}+cm, & \qquad 
	|\tau_g-\tau_e|\leq c\eta R_e^{-2}+cmR_e^{-1}.
	\end{align*}
	Moreover, let $\overline{R_S}:=\phi^2(R_e)R_e$, then
	$$
	|\overline{R_S}-R_S|\leq cm(\tau_e+|\Acirc_e|_{L^2(\Sigma)}).
	$$

	\label{ more comp}
\end{lem}
Another useful tool is an integrated version of the Gauss equation. More precisely, any sphere $\Sigma$ satisfies
\begin{align}
16\pi-\int_\Sigma H^2=-2\int_\Sigma |\Acirc|^2d\mu-4\int_\Sigma\operatorname{G}(\nu,\nu)d\mu, \label{gausseqn}
\end{align}
where $G=\operatorname{Rc}-\frac12 \gamma \operatorname{Sc}$ is the Einstein tensor. It follows from (\ref{gausseqn}), Lemma \ref{r compare} and (\ref{rc schwarzschild})  that any admissible surface enjoys  uniform bounds on $|H|_{L^2(\Sigma)}$ and $|A|_{L^2(\Sigma)}$, that is,
\begin{align}
|H|_{L^2(\Sigma)}, |A|_{L^2(\Sigma)}\leq c(\epsilon,\delta). \label{l2 curvature bound}
\end{align}
Finally, we need the Michael-Simon-Sobolev inequality which can, for instance, be found as Proposition 5.4 in \cite{huisken1996definition}.
\begin{lem}
	If $r_{\min}\geq R_0$ for some $R_0(\eta,m)$ sufficiently large, then any smooth function $f$ satisfies
	$$
	\int_\Sigma |f|^2d\mu \leq c \bigg(\int_\Sigma (|\nabla f|+|fH|)d\mu\bigg)^2.
	$$
	\label{michael simon sobolev}
\end{lem}
 We finish this section by showing that an area preserving Willmore flow can only cease to be admissible if its barycentre moves outwards. We also prove an estimate for the excess Willmore energy.
\begin{lem}
	Let $\Sigma_t, t\in[0,T),$ be an admissible Willmore flow starting at a surface $\Sigma_0=\Sigma$ which satisfies $|\Acirc|_{L^2(\Sigma)}<\epsilon/2, \tau_e<\delta/2$ and $r_{\min}\geq NR_0$ for some constant $N>1$. If $R_0$ and $N$ are sufficiently large, then the following  estimates hold:
	$$
	|\Acirc|^2_{L^2(\Sigma_t)}(t)\leq \frac34 \epsilon,
\qquad
	\mathcal{W}(t)-\mathcal{W}(0)\leq \epsilon, \qquad
	r_{\min}\geq 2R_0.
	$$
	\label{excess estimate}
\end{lem}
\begin{proof}
Lemma \ref{r compare} and Lemma \ref{ more comp} imply the crude estimate $\int_{\Sigma} G(\nu,\nu)d\mu\leq cR_e^{-1}$. Using (\ref{gausseqn}) and the fact that the flow decreases the Willmore energy we obtain
\begin{align*}
0&\geq |\Acirc|_{L^2(\Sigma_t)}^2-|\Acirc|_{L^2(\Sigma_0)}^2-cR_e^{-1}.
\end{align*}
Hence, the claim follows if $R_0$ is sufficiently large such that $cR_e^{-1}\leq cR_0^{-1}<\epsilon/4$. For the second claim we trivially estimate
\begin{align*}
\mathcal{W}(0)-\mathcal{W}(t)\leq& |\Acirc|_{L^2(\Sigma_0)}^2+cR_e^{-1}<\epsilon.
\end{align*}
 For the last claim, we note that according to Lemma \ref{r compare} and Lemma \ref{ more comp} there holds $r_{\min}(t)\geq c^{-1} R_g(t)$ as well as
$R_g(0)\geq c^{-1} r_{\min}(0)\geq c^{-1}N R_0$. Since $R_g(t)=R_g(0)$ the claim now follows if $N$  is chosen sufficiently large.
\end{proof}
\begin{rema}
	According to a classical inequality by Li and Yau (\cite{li1982new}) embeddedness is automatically implied by the smallness of $|\Acirc|^2_{L^2(\Sigma_t)}$.
\end{rema}
It follows from the previous lemma that a flow can only cease to be admissible if $\tau_e$ reaches $\delta$. In order to study the evolution of $\tau_e$ we need to to establish precise curvature estimates. 
\section{Integral curvature estimates}
In this section, we prove general curvature estimates for embedded spheres with small traceless part of the second fundamental form in the $L^2-$sense. In the next section, we will use these to obtain precise estimates for the evolution of certain geometric quantities. However, this section might be of independent interest, too. Unless otherwise stated, we only assume that $\Sigma$ is spherical, that the conclusion of Lemma \ref{r compare} holds, that $r_{\min}\geq R_0$  and that the estimate (\ref{l2 curvature bound}) is valid. We will then state all estimates in terms of $R_e$.  As before, we denote the connection of $\Sigma$ by $\nabla$ and the connection of the ambient space by $\overline{\nabla}$. We also use the common $*$-notation to summarize geometric terms. In order to obtain estimates for the different components of the second fundamental form, we follow the approach of \cite[section 2]{kuwert2001willmore}. However, we need to take the effect of the ambient curvature into account. We need the following lemma which follows from a straight-forward computation:
\begin{lem}
	\label{der swap}
	Let $\psi\in \Omega^k(\Sigma)$ for some $k\in\mathbb{N}$, $p\in \Sigma$ and $e_1,e_2$ be and orthonormal frame of $\Sigma$ at $p$. Let $X_i\in\{e_1,e_2\}$, $1\leq i\leq k$. There holds
	\begin{align*}
	&(\nabla\nabla^*\psi-\nabla^*\nabla\psi)(X_1,...,X_k)\\=&\psi(\operatorname{Rm^\Sigma}(X_1,e_i,e_i),X_2,...,X_k)
	+\sum_{j=2}^{k}\psi(e_i,\dots,\operatorname{Rm}^\Sigma(X_1,e_i,X_j),...,X_k)-\nabla^*T_\psi(X_1,...,X_k),
	\end{align*}
	where $\nabla^*=-\operatorname{div}$ is the adjoint of $\nabla$ and
	$$
	T_\psi(X_0,..X_k)=\nabla_{X_0}\psi(X_1,..,X_k)-\nabla_{X_1}\psi(X_0,..,X_k).
	$$
\end{lem}
We would now like to use the previous lemma to express certain geometric quantities in terms of the Willmore operator $W$ defined by $W=\Delta H+H\operatorname{Rc}(\nu,\nu)+H|\Acirc|^2$. To this end, we first need the following lemma.
\begin{lem}
	The following identities hold:
	\begin{align}
	\Delta \Acirc=&(\nabla ^2 H-\frac{1}{2}\gamma\Delta H)+\frac12 H^2\Acirc+\Acirc*\Acirc*\Acirc+\operatorname{Rm}*A+\overline{\nabla}{\operatorname{Rm}}*1,
	\label{Simon}
	\\
	\nabla^*\nabla^2 H
	=&-\nabla(\Delta H)+{\operatorname{Rm}}*\nabla H+\Acirc*\Acirc*\nabla \Acirc+{\operatorname{Rm}}*\Acirc*\Acirc -\frac14 H^2\nabla H,
	\label{identity2}
	\\
	\nabla^*\nabla^2 H=& -\nabla(\Delta H)+{\operatorname{Rm}}*\nabla H+\Acirc*\Acirc*\nabla H-\frac14H^2\nabla H,
	\label{identity2b}
	\\
	\nabla^*(\nabla^2\Acirc)=&\nabla(\nabla^*\nabla \Acirc)+\nabla \Acirc*A*A+\overline{\nabla}{\operatorname{Rm}}* \Acirc
	+{\operatorname{Rm}}*\Acirc*A+\operatorname{Rm}*\nabla \Acirc.
	\label{identity 3}
	\end{align}	
\end{lem}
\begin{proof}
	With the convention
	$$
	A_{ij}=A(X_i,X_j)=g(\overline{\nabla}_{X_i}X_j,\nu)
	$$
	the orthonormal frame satisfies
	\begin{align}
	\overline{\nabla}_{X_i}X_j=\nabla_{X_i}X_j+A_{ij}\nu=A_{ij}\nu. \label{ambient vs intrinsic derivative}
	\end{align}
	Chosing $\psi=A$ in Lemma \ref{der swap} we obtain
	$$
	T_A(X_1,X_2,X_3)=X_1(g(\bar{\nabla}_{X_2}X_3,\nu))-X_2(g(\bar{\nabla}_{X_1}X_3,\nu))={\operatorname{Rm}}(X_1,X_2,X_3,\nu)
	$$
	as $\overline{\nabla}\nu$ is tangential. We clearly have $\nabla^*T(X_1,X_2)=X_i({\operatorname{Rm}}(X_i,X_1,X_2,\nu))$ which implies
	$$
	\nabla^*T_A=\overline{\nabla}{\operatorname{Rm}}*1+{\operatorname{Rm}}*A. 
	$$
	A straight forward calculation gives
	$$
	\nabla^*A=-\nabla H +{\operatorname{Rm}}*1.
	$$
	Evaluated at $(X_1,X_2)$, this yields (with slight abuse of notation)
	\begin{align}
	\Delta A=\nabla^2 H+A(\operatorname{Rm}^\Sigma(X_1,e_i,e_i),X_2)+A(e_i,\operatorname{Rm}^\Sigma(X_1,e_i,X_2))+\overline{\nabla}{\operatorname{Rm}}*1+{\operatorname{Rm}}*A.
	\label{bt eq 1}
	\end{align}
	If we furthermore assume that the $e_i$ are principal directions, we obtain
	\begin{align*}
	&A(\operatorname{Rm}^\Sigma(X_1,e_i,e_i),X_2)+A(e_i,\operatorname{Rm}^\Sigma(X_1,e_i,X_2))
	\\=&	\operatorname{Rm}^\Sigma(X_1,e_i,e_i,X_2)A(X_2,X_2)+\operatorname{Rm}^\Sigma(X_1,e_i,X_2,e_i)
	A(e_i,e_i).
	\end{align*}
	If $X_1, X_2$ are distinct, both terms vanish. Otherwise we can assume that $X_1=X_2=e_1$. Then, using the Gauss equation
	$$
	\operatorname{Rm}^\Sigma(a,b,c,d)={\operatorname{Rm}}(a,b,c,d)+A(a,d)A(b,c)-A(a,c)A(b,d)
	$$
	the right hand side becomes (again with abuse of notation)
	\begin{align*}
	&\operatorname{Rm}^\Sigma(e_1,e_2,e_2,e_1)(A(e_1,e_1)-A(e_2,e_2))
	\\=&{\operatorname{Rm}}*A+A_{11}A_{22}(A_{11}-A_{22}).
	\\=&{\operatorname{Rm}}*A+(\Acirc_{11}+\frac12H)(\Acirc_{22}+\frac12H)(\Acirc_{11}-\Acirc_{22})
	\\=&{\operatorname{Rm}}*A-2(\Acirc_{11}+\frac12 H)(\Acirc_{11}-\frac12 H)\Acirc_{11},
	\end{align*}
	where we used that $\Acirc$ is trace free. Together with (\ref{bt eq 1}) this clearly implies
	\begin{align*}
	\Delta A=\nabla ^2 H+\frac12 H^2\Acirc+\Acirc*\Acirc*\Acirc+{\operatorname{Rm}}*A+\overline{\nabla}{\operatorname{Rm}}*1.
	\end{align*}
	As $\nabla \gamma=0$, there holds $\Delta(\gamma H)=\gamma\Delta H$ and we obtain the first claim. Choosing $\psi=\nabla H$ we find $T_{\nabla H}=0$ by the symmetry of second derivatives and evaluating at $X_1=e_1$ we obtain (again, with abuse of notation)
	\begin{align*}
	\nabla^*\nabla^2 H=&\nabla \nabla^* \nabla H-\operatorname{Rm}^\Sigma(X_1,e_i,e_i,e_j)\nabla H(e_j)
	\\=&\nabla \nabla^* \nabla H- \operatorname{Rm}^\Sigma(e_1,e_2,e_2,e_1)e_1(H)
	\\=& -\nabla(\Delta H)+{\operatorname{Rm}}*\nabla H+(\Acirc_{11}^2-\frac14 H^2)e_1(H)
	\\=&-\nabla(\Delta H)+{\operatorname{Rm}}*\nabla H+\Acirc*\Acirc*\nabla^*\Acirc+{\operatorname{Rm}}*\Acirc*\Acirc -\frac14 H^2\nabla H,
	\end{align*}
	as $\nabla^* \Acirc=-\frac12 \nabla H +{\operatorname{Rm}}*1$. This implies the second and third claim. Finally, if $\psi=\nabla \Acirc$, then at $(X_1,X_2,X_3,X_4)$ we have
	$$
	T_{\nabla \Acirc}=\Acirc(\operatorname{Rm}^\Sigma (X_1,X_2,X_3),X_4)+\Acirc(X_3,\operatorname{Rm}^\Sigma(X_1,X_2,X_4))
	={\operatorname{Rm}}*\Acirc+\Acirc*A*A,
	$$
	hence
	$$
	\nabla^*T_{\nabla \Acirc}=\overline{\nabla}{\operatorname{Rm}}* \Acirc
	+{\operatorname{Rm}}*\Acirc*A+\operatorname{Rm}*\nabla \Acirc+\nabla \Acirc * A * A.
	$$
	This gives the very rough identity
	\begin{align*}
	\nabla^*(\nabla^2\Acirc)=\nabla(\nabla^*\nabla \Acirc)+\nabla \Acirc*A*A+\overline{\nabla}{\operatorname{Rm}}* \Acirc
	+{\operatorname{Rm}}*\Acirc*A+\operatorname{Rm}*\nabla \Acirc.
	\end{align*}
\end{proof}
\begin{rema}
	In fact, one can actually show the more precise identity
	\begin{align}
	\Acirc \cdot \nabla^2 \Acirc = \Acirc \cdot \nabla ^2 H+\frac12 H^2 |\Acirc|^2-|\Acirc|^4+\Acirc*\Acirc*\operatorname{Rm}+2\Acirc\cdot \nabla \operatorname{Rc}(\nu,\cdot)^T,
	\label{simonsidentity}
	\end{align}
	see (1.7) in \cite{lamm2011foliations}. This can be seen by using the identity $\nabla^* \Acirc=-\frac12 \nabla H +\operatorname{Rc}^T(\cdot,\nu)$.
\end{rema}
For the following proof we remark that $|\operatorname{Rc}|\leq c\eta R_e^{-4}+cmR_e^{-3}\leq cR_e^{-3}$ and $|\overline{\nabla}\operatorname{Rc}|\leq c\eta R_e^{-5}+cmR_e^{-4}\leq cR_e^{-4}$. Moreover, since $|H_e|^2_{L^2(\Sigma)}\geq 16\pi$, there evidently holds $|A|^2_{L^2(\Sigma)}\geq 8\pi$, provided $r_{\min}\geq R_0$ (c.f. Lemma \ref{identities}). 
\begin{lem}
	\label{gradestimate}
	Let $\Sigma$ be a spherical surface satisfying $r_{\min}\geq R_0$, the conclusion of Lemma \ref{r compare} and (\ref{l2 curvature bound}).	Then, there holds
	$$
	\int_{\Sigma} |\nabla A|^2 d\mu \leq -c\int_{\Sigma} HWd\mu+\frac{c}{R_e^3}\int_{\Sigma} |A|^2d\mu+c\int_\Sigma |\Acirc|^4 d\mu,
	$$
	provided $R_0$ is chosen to be sufficiently large. 	In this estimate, we may replace $\nabla A$ by $\nabla \Acirc$ or $\nabla H$.
	\label{section 4 lemma with assumptions}	
	\begin{proof} 
		Multiplying (\ref{Simon}) by $\Acirc$, integrating and using that $\Acirc$ is trace free we obtain
		\begin{align*}
		\int_{\Sigma} |\nabla \Acirc|^2 d\mu +\frac12\int_{\Sigma} H^2 |\Acirc|^2d\mu &\leq - \int_{\Sigma} \nabla H\cdot \nabla^* \Acirc d\mu +c\int_{\Sigma} |\Acirc|^4 d\mu +\frac{c}{R_e^3} \int_{\Sigma} |A||\Acirc| d\mu +\frac{c}{R_e^4}\int_{\Sigma} |\Acirc|d\mu \\&\leq 
		\frac12 \int_{\Sigma} |\nabla H|^2d\mu +\frac{c}{R_e^3}\int_{\Sigma} |\nabla H|d\mu+c\int_{\Sigma} |\Acirc|^4d\mu+\frac{c}{R_e^3} \int_{\Sigma} |A|^2d\mu,
		\end{align*}
		where we used $\nabla H=-\nabla^* \Acirc+\operatorname{Rm}*1$, $|\Acirc||A|\leq R_e^{-1}|A|^2+R_e|\Acirc|^2$ and estimated
		$$
		\int_\Sigma |\Acirc|d\mu\leq c\int_{\Sigma} R_e|\Acirc|^2d\mu+cR_e^{-1}\int_\Sigma 1d\mu\leq  c R_e^3 \int_\Sigma |\Acirc|^4+cR_e  \leq c R_e^3 \int_\Sigma |\Acirc|^4d\mu+ c R_e \int_{\Sigma} |A|^2d\mu.
		$$
		Integrating by parts and using the definition of $W$ we obtain
		\begin{align*}
		\frac 12\int_{\Sigma} |\nabla H|^2d\mu =& \frac12 \int_{\Sigma} ( -HW+H^2{\operatorname{Rc}}(\nu,\nu)+H^2|\Acirc|^2)d\mu \\\leq& 
		\frac12 \int_{\Sigma} ( -HW+H^2|\Acirc|^2)d\mu+cR_e^{-3}\int_{\Sigma} |A|^2 d\mu.
		\end{align*}
		Next, $|\nabla H|\leq c|\nabla \Acirc|+cR_e^{-3}$ and hence 
		$$
		\int_{\Sigma} |\nabla H|R_e^{-3}d\mu\leq c R_e^{-1} \int_{\Sigma} |\nabla \Acirc|^2d\mu +cR_e^{-3}\int_{\Sigma} |A|^2d\mu.
		$$
		The claimed inequality now follows from combining these estimates,  absorbing the $|\nabla \Acirc|^2$ term and using $|\nabla A|^2\leq c |\nabla \Acirc|^2+cR_e^{-6}$. Finally, we clearly have $|\nabla H|\leq c |\nabla A|$.	
	\end{proof}
\end{lem}
\begin{lem}
	Under the assumptions of the previous lemma, there holds
	\begin{align*}
	\int_{\Sigma} |\nabla ^2H|^2d\mu+\int_{\Sigma} |A|^2|\nabla A|^2d\mu+\int_{\Sigma} |A|^4|\Acirc|^2d\mu \\ \leq c\int_{\Sigma} |W|^2 d\mu -cR_e^{-3} \int_{\Sigma} WHd\mu+c\int_{\Sigma} (|\Acirc|^2|\nabla \Acirc|^2+|\Acirc|^6) d\mu +cR_e^{-6} \int_{\Sigma} |A|^2 d\mu.
	\end{align*}
	
\end{lem}
\begin{proof}
	Multiplying (\ref{identity2}) with $\nabla H$ and integrating we obtain
	
	\begin{align*}
	\int_{\Sigma} |\nabla^2H|^2d\mu+\frac 14\int_{\Sigma} H^2|\nabla H|^2 d\mu \leq&  \int_{\Sigma} |\Delta H|^2d\mu + cR_e^{-3}\int_{\Sigma} |\nabla H|^2d\mu +c\int_{\Sigma} |\Acirc|^2|\nabla \Acirc||\nabla H|d\mu\\&+cR_e^{-3}\int_{\Sigma} |\Acirc|^2|\nabla H|d\mu.
	\end{align*}
	The last term can be estimated by
	$$
	c(\kappa)R_e^{-6} \int_{\Sigma} |A|^2 d\mu +\kappa/2 \int_{\Sigma} |\nabla A|^2|A|^2|d\mu.
	$$
	Next, we have
	$$
	\int_{\Sigma} |\Delta H|^2 d\mu \leq c \int_{\Sigma} |W|^2 d\mu+cR_e^{-6} \int_{\Sigma} |A|^2 d\mu + c \int_{\Sigma} |\Acirc|^4|A|^2 d\mu 
	$$
	and 
	$\int_{\Sigma} |\Acirc|^4|A|^2d\mu \leq c(\kappa) \int_{\Sigma} |\Acirc|^6d\mu+\kappa \int_{\Sigma} |A|^4|\Acirc|^2d\mu$. There also holds
	$$
	\int_{\Sigma} |\Acirc|^2|\nabla \Acirc||\nabla H|d\mu \leq \kappa/2 \int_{\Sigma} |A|^2|\nabla A|^2 d\mu +c(\kappa)\int_{\Sigma} |\nabla \Acirc|^2|\Acirc|^2d\mu
	$$
	and the term $R_e^{-3}\int_\Sigma |\nabla H|^2 d\mu$ can be estimated with the previous lemma and the trivial estimate $\int_{\Sigma} |\Acirc|^4d\mu\leq 2 R_e^3 \int_{\Sigma}   |\Acirc|^6d\mu+2R_e^{-3}\int_{\Sigma} |A|^2d\mu$. Hence, so far we have shown that
	\begin{align}
	\int_{\Sigma} |\nabla^2H|^2d\mu+\frac 14\int_{\Sigma} H^2|\nabla H|^2 d\mu
	\leq 
	c \int_{\Sigma} |W|^2 d\mu+c(\kappa )R_e^{-6} \int_{\Sigma} |A|^2 d\mu + \kappa \int_{\Sigma} |\Acirc|^2|A|^4 d\mu  \notag \\
	+\kappa  \int_{\Sigma} |A|^2|\nabla A|^2 d\mu +c(\kappa)\int_{\Sigma} |\nabla \Acirc|^2|\Acirc|^2d\mu 
	-cR_e^{-3}\int_{\Sigma} WHd\mu+c(\kappa)\int_\Sigma |\Acirc|^6 d\mu \notag.
	\end{align}
	From the estimate $|\nabla A|^2\leq c |\nabla \Acirc|^2+cR_e^{-6}$ and (\ref{Simon}) it follows that 
	\begin{align*}
	\frac1c	\int_{\Sigma} H^2 |\nabla A|^2d\mu-cR_e^{-6}\int_{\Sigma}|A|^2d\mu\leq& \int H^2|\nabla \Acirc|^2d\mu
	\\ =&-\int_{\Sigma} H^2 \Acirc\cdot \nabla^2Hd\mu -\frac 12 \int H^4|\Acirc|^2 d\mu
	\\&+\int_{\Sigma} (H^2*\Acirc*\Acirc*\Acirc*\Acirc+H^2*\Acirc*({\operatorname{Rm}}*A+\overline{\nabla}{\operatorname{Rm}})d\mu\\&+\int_{\Sigma} H*\nabla H*\Acirc*\nabla \Acirc)d\mu.
	\end{align*}
	The second and third row can be estimated by
	$$
	\kappa \int_{\Sigma} |A|^4 |\Acirc|^2d\mu +c(\kappa) \int_{\Sigma} |\Acirc|^6 d\mu +c(\kappa)R_e^{-6}\int_{\Sigma} |A|^2d\mu +\kappa \int_{\Sigma} |A|^2|\nabla A|^2d\mu +\int_{\Sigma} c(\kappa)|\Acirc|^2|\nabla \Acirc|^2d\mu.
	$$
	Using partial integration and the identity $-\nabla^*\Acirc=\frac12 \nabla H+\operatorname{Rm}*1$, the first term in the first row can be computed to be
	$$
	-\int_{\Sigma} H^2\Acirc\nabla^2Hd\mu = \frac 12 \int_{\Sigma} H^2|\nabla H|^2 d\mu + \int_{\Sigma} H^2*\nabla H *{\operatorname{Rm}}d\mu + \int_{\Sigma} \nabla H*\nabla H * \Acirc * Hd\mu.
	$$
	In the last equation, the third term can be estimated as before and the second term can be estimated by
	$$
	\frac{c(\kappa)}{R_e^6} \int_{\Sigma} |A|^2d\mu+\kappa\int_{\Sigma}|A|^2|\nabla A|^2d\mu.
	$$
	Finally, we note that
	$$
	\int_{\Sigma} A^2|\nabla A|^2d\mu \leq c\int_{\Sigma} H^2|\nabla A|^2d\mu+c\int_{\Sigma} |\Acirc|^2|\nabla \Acirc|^2d\mu+\frac{c}{R_e^6}\int_{\Sigma} |A|^2d\mu.
	$$
	Combining all these inequalities, choosing $\kappa>0$ sufficiently small, absorbing all terms when possible and noting that
	$$
	\int_{\Sigma} |A|^4|\Acirc|^2d\mu \leq \int_{\Sigma} H^4|\Acirc|^2d\mu+\int_{\Sigma}|\Acirc|^6d\mu
	$$ we obtain the desired statement.
\end{proof}
\begin{lem}
	\label{absorbing}
	Under the assumptions of Lemma \ref{section 4 lemma with assumptions} there holds
	\begin{align*}
	&\int_{\Sigma} |\nabla ^2 \Acirc|^2d\mu +\int_{\Sigma} |A|^4|\Acirc|^2d\mu+\int_{\Sigma} |\nabla A|^2|A|^2d\mu \\\leq& c \int_{\Sigma} |W|^2d\mu +c\int_{\Sigma} (|\nabla \Acirc|^2|\Acirc|^2+|\Acirc|^6)d\mu +cR_e^{-6}\int_{\Sigma} |A|^2d\mu,
	\end{align*}
	provided $R_0$ is sufficiently large.
\end{lem}
\begin{proof}
	Multiplying (\ref{identity 3}) by $\nabla \Acirc$ and applying (\ref{Simon}) we obtain
	\begin{align*}
	\int_{\Sigma} |\nabla^2 \Acirc|^2d\mu\leq& \int_{\Sigma} |\Delta \Acirc|^2+c\int_{\Sigma} |\nabla \Acirc|^2|A|^2d\mu+cR_e^{-4}\int_{\Sigma} |\Acirc||\nabla \Acirc|d\mu \\&+cR_e^{-3}\int_{\Sigma} |\Acirc||A||\nabla \Acirc|d\mu +cR_e^{-3}\int_{\Sigma} |\nabla \Acirc|^2d\mu
	\\\leq& c\int_{\Sigma} |\nabla^2 H|^2d\mu +c\int_{\Sigma} |A|^4|\Acirc|^2 d\mu +c\int |\Acirc|^6d\mu +c\int_{\Sigma} |\nabla A|^2|A|^2d\mu\\&+\frac{c}{R_e^6}\int_{\Sigma} |A|^2d\mu+\frac{c}{R_e^3}\int_{\Sigma} |\nabla A|^2d\mu,
	\end{align*}
	where we used $|\Acirc||\nabla \Acirc|\leq c|A||\nabla A|$.
	The claim now follows from the two previous lemmas and from estimating
	$$
	R_e^{-3}\int_{\Sigma} |HW|d\mu\leq 2R_e^{-6}\int_{\Sigma} |A|^2d\mu+2\int_{\Sigma} |W|^2d\mu.
	$$
\end{proof}
We now need the following Sobolev-type interpolation inequality.
\begin{lem}
	\label{sobolev}
	Let $\Sigma$ be a compact surface satisfying $r_{\min}\geq R_0$. If	$R_0$ is chosen sufficiently large, then any smooth $k-$form $\psi$ satisfies
	$$
	|\psi|^4_{L^\infty(\Sigma)}\leq c|\psi|^2_{L^2(\Sigma)}\bigg(\int_{\Sigma} (|\nabla^2\psi|^2+|\psi|^2H^4)d\mu\bigg).
	$$
	\begin{proof} The proof of Theorem 5.6. in \cite{kuwert2002gradient} caries over to our setting as $\Sigma$ is compact and since the Michael-Simon-Sobolev inequality holds in an asymptotic region of an asymptotically Schwarzschild manifold, see Lemma \ref{michael simon sobolev}. One then easily adapts the proof of Lemma 2.8 in \cite{kuwert2001willmore}.
	\end{proof}
\end{lem}
We also need the following multiplicative Sobolev-inequality.
\begin{lem}
	Under the assumptions of the previous lemma we have
	$$
	\int_{\Sigma} (|\Acirc|^2|\nabla \Acirc|^2+|\Acirc|^6)d\mu \leq c \int_{\Sigma} |\Acirc|^2d\mu \int_{\Sigma} (|\nabla^2 \Acirc|^2 +|A|^2|\nabla A|^2+|A|^4|\Acirc|^2)d\mu.
	$$
\end{lem}
\begin{proof}
	This is a straight forward adaption of Lemma 2.5 in \cite{kuwert2001willmore}. Again, the proof only relies on the Michael-Simon-Sobolev inequality, Young's inequality and H\"older's inequality.
\end{proof}
At this point, a small curvature assumption allows us to absorb the term on the left hand side of the previous equation in Lemma \ref{absorbing}. This finally yields an $L^\infty$-estimate for $\Acirc$.
\begin{lem}
	\label{linfty}
	Under the assumptions of Lemma \ref{section 4 lemma with assumptions} there exists a constant $\epsilon(m,\eta, R_0)>0$ such that if
	$$
	\int_{\Sigma} |{\Acirc}|^2d\mu\leq \epsilon,
	$$
	then
	\begin{align}
	\int_{\Sigma} |\nabla ^2\Acirc|^2d\mu+\int_{\Sigma}|A|^2|\nabla A|^2d\mu+\int_{\Sigma} |A|^4|\Acirc|^2d\mu
	\leq c\int_{\Sigma} |W|^2d\mu+cR_{e}^{-6}\int_{\Sigma} |A|^2d\mu.
	\label{b2ndder estimate}
	\end{align}
	In particular, we have
	$$
	|\Acirc|_{L^\infty(\Sigma)}^4\leq c|\Acirc|^2_{L^2(\Sigma)}\bigg( \int_{\Sigma} |W|^2d\mu+R_e^{-6}\int_\Sigma |A|^2 d\mu\bigg).
	$$
\end{lem}
\begin{proof}
	This is a direct consequence of the previous three lemmas.
\end{proof}
It turns out that the integral curvature estimates also imply an improved estimate for $|\Acirc|_{L^2(\Sigma)}$.
\begin{lem}
	Under the assumptions of Lemma \ref{section 4 lemma with assumptions} there exists a constant $\epsilon(m,\eta,R_0)>0$ such that if
	$$
	\int_{\Sigma} |{\Acirc}|^2d\mu\leq \epsilon,
	$$
	then
	$$
	\int_\Sigma |\Acirc|^2 d\mu \leq cR_e^2 \int_\Sigma (|\nabla H|^2+|\operatorname{Rc}(\nu,\cdot)^T|^2) d\mu.
	$$
	\label{improved BL2}
\end{lem}
\begin{proof}
	Integrating (\ref{simonsidentity}) and using integration by parts as well as the identity $\nabla^* \Acirc=-\frac12 \nabla H +\operatorname{Rc}(\nu,\cdot)^T$ yields
	$$
	\int_\Sigma (|\nabla \Acirc|^2+H^2|\Acirc|^2) d\mu  \leq c \int_\Sigma (|\nabla H|^2 + |\Acirc|^4+|\operatorname{Rc}(\nu,\cdot)^T|^2+R_e^{-3}|\Acirc|^2)d\mu.	
	$$
	Next, using Lemma \ref{michael simon sobolev} and H\"older's inequality, we obtain
	$$
	\int_\Sigma |\Acirc|^4 d\mu \leq |\Acirc|_{L^2(\Sigma)}^2\bigg( \int_\Sigma (|\nabla \Acirc|^2+H^2|\Acirc|^2)d\mu\bigg).
	$$
	Hence, this term can be absorbed. Again with the Michael-Simon-Sobolev inequality and H\"older's inequality we get
	$$
	\int_\Sigma |\Acirc|^2 d\mu \leq c |\Sigma|\int_\Sigma (|\nabla \Acirc|^2+H^2|\Acirc|^2) d\mu.
	$$
	From this the claim follows.
\end{proof}
Finally, we prove two useful $W^{2,2}-$ and $L^\infty-$estimates for the mean curvature.
\begin{lem}
	Under the assumptions of Lemma \ref{section 4 lemma with assumptions} there holds
	$$
	\int_{\Sigma_t} |\nabla^2 H|^2d\mu +\frac14 \int_{\Sigma_t} H^2|\nabla H|^2d\mu \leq c(|\Delta H|_{L^2(\Sigma_t)}^2+|\Acirc|^4_{L^\infty(\Sigma_t)}+R_e^{-6}).
	$$
	\label{bochner}
\end{lem}
\begin{proof}
	We multiply (\ref{identity2b}) by $\nabla H$ to obtain the Bochner-type identity
	\begin{align*}
	\int_{\Sigma_t} |\nabla^2 H|^2d\mu +\frac14 \int_{\Sigma_t} H^2|\nabla H|^2d\mu \leq & \int_{\Sigma_t} |\Delta H|^2d\mu+\int_{\Sigma_t} |\overline{\operatorname{Rm}}*\nabla H * \nabla H| d\mu +\int_{\Sigma_t} |\Acirc|^2|\nabla H|^2d\mu \\
	\leq & |\Delta H|_{L^2(\Sigma_t)}^2+c\epsilon(R_e^{-3}+|\Acirc|^2_{L^\infty(\Sigma_t)})|\Delta H|_{L^2(\Sigma_t)} \\
	\leq& c(|\Delta H|_{L^2(\Sigma_t)}^2+|\Acirc|^4_{L^\infty(\Sigma_t)}+R_e^{-6}).
	\end{align*}
\end{proof}
To state the $L^\infty$-estimate we introduce the quantity $\overline{H_S}=2R_e^{-1}\phi^{-2}(R_e)-2mR_e^{-2}\phi^{-3}(R_e)$ to be the mean curvature, with respect to the Schwarzschild metric, of a centred sphere with Euclidean radius $R_e$.
\begin{lem}
	Under the assumptions of Lemma \ref{section 4 lemma with assumptions} and provided $R_0^{-1}, |\Acirc|^2_{L^2(\Sigma))}<\epsilon$ there holds
	\label{H linfty estimate}
	$$
	|H-\overline{ H_S}|^4_{L^\infty(\Sigma_t)}\leq c(|\Acirc|^2_{L^2(\Sigma_t)}+R_e^{-2})(|\Delta H|_{L^2(\Sigma_t)}^2+|\Acirc|^4_{L^\infty(\Sigma_t)}+R_e^{-6}).
	$$
\end{lem}
\begin{proof}
	Using Lemma \ref{some estimates} below we obtain $|H-\overline{ H_S}|^2_{L^2(\Sigma_t)}\leq c(|\Acirc|^2_{L^2(\Sigma_t)}+R_e^{-2})\leq c\epsilon$. Now, Lemma \ref{sobolev} and Lemma \ref{bochner} as well as $|\overline{ H_S}|\leq cR_e^{-1}$ imply that 
	\begin{align*}
	&|H-\overline{ H_S}|^4_{L^\infty(\Sigma_t)}\\\leq& c|H-\overline{ H_S}|^2_{L^2}\bigg(\int_{\Sigma_t} |\nabla^2 H|^2d\mu +\int_{\Sigma_t} H^4|H-\overline{ H_S}|^2d\mu \bigg)
	\\\leq &c(|\Acirc|^2_{L^2(\Sigma_t)}+R_e^{-2})(|\Delta H|_{L^2(\Sigma_t)}^2+|\Acirc|^4_{L^\infty(\Sigma_t)}+R_e^{-6})\\&+c(|\Acirc|^2_{L^2(\Sigma_t)}+R_e^{-2})\bigg(R_e^{-4}\int_{\Sigma_t} |H-\overline{ H_S}|^2d\mu +|H-\overline{ H_S}|^4_{L^\infty(\Sigma_t)}\int_{\Sigma_t} |H-\overline{ H_S}|^2d\mu \bigg).
	\end{align*}
	The third term can be absorbed using (\ref{h l2 est}) and estimating $|\Acirc|^2_{L^2(\Sigma_t)}\leq \epsilon$. Another application of (\ref{h l2 est}) yields
	$$
	|H-\overline{ H_S}|^2_{L^2(\Sigma_t)}\leq |\Acirc|^2_{L^2(\Sigma_t)}+cR_e^{-2} \leq c R_e^{4 }|\Acirc|^4_{L^\infty(\Sigma_t)}+cR_e^{-2}. 
	$$
	Using this to estimate the second term, the claim follows. 
\end{proof}
\section{A-priori Estimates}
In this section, we specify the estimates from the previous section to the situation of an area-preserving Willmore flow. We first prove the following useful lemma.
\begin{lem}
	Let $\Sigma_t$ be an admissible area preserving Willmore flow. Then every $\Sigma_t$ satisfies
	\begin{align}
	\int_{\Sigma_t} |\operatorname{Rc}(\nu,\cdot)^T|^2 d\mu \leq c R_e^{-4}(R_e|\Acirc|^2_{L^2(\Sigma_t)}+\tau_e^2+ \eta R_e^{-1}), 
	\label{rct est}\\
	|H-\overline{ H_S}|_{L^2(\Sigma_t)} \leq c (|\Acirc|_{L^2(\Sigma_t)}+m\tau_e R_e^{-1}+\eta R_e^{-2}). \label{h l2 est}
	\end{align}
	\label{some estimates}Here, $\overline{H}_S:=2R_e^{-1}\phi^{-2}(R_e)-2mR_e^{-2}\phi^{-3}(R_e)$.
	\begin{proof}
		Using Lemma \ref{identities}, we may assume that $\eta =0$. Let $e_i$ be an orthonormal frame of $T_p\Sigma_t$ at a point $p$. There holds
		$$
		m^2r^6\phi^4 |\operatorname{Rc}(\nu,\cdot)^T|^2=9 (\partial_r\cdot \nu_e)^2 \partial_r \cdot e_i \partial_r \cdot e_i =9 (\partial_r\cdot \nu_e)^2 \partial_r\cdot (\partial_r-\nu_e\cdot \partial_r \nu_e).
		$$
		We would like to express the right hand side in terms of the approximating sphere $S$ introduced in Lemma \ref{muldelel}. Denoting the respective quantities regarding $S$ with a tilde we find using Lemma \ref{muldelel}
		\begin{align}
		&\bigg|\int_S \tilde r^{-6}(\tilde\partial_r\cdot \tilde\nu_e)^2 \tilde\partial_r\cdot (\tilde\partial_r-\tilde\nu_e\cdot \tilde\partial_r \tilde\nu_e)d\tilde\mu_e-\int_\Sigma r^{-6}(\partial_r\cdot \nu_e)^2 \partial_r\cdot (\partial_r-\nu_e\cdot \partial_r \nu_e) d\mu_e \bigg|\notag\\\leq& cR_e^{-4}|\Acirc_e|^2_{L^2(\Sigma_t)}.
		\end{align}
		By Young's inequality and Lemma \ref{identities} this error can be further estimated via
		$$
		R_e^{-4}|\Acirc_e|_{L^2(\Sigma_t)}
		\leq cR_e^{-5}+cR_e^{-3}|\Acirc|^2_{L^2(\Sigma_t)}.
		$$
		Now, for a round sphere, there holds
		$\tilde\partial_{ r}=\tilde r^{-1}(a_e+R_e\tilde\nu_e)$. Hence,
		$$
		(\tilde\partial_{ r}-\tilde\partial_{ r}\cdot\tilde\nu_e \tilde \nu_e)\cdot \partial_{\tilde r}
		=\tilde r^{-2}(a_e+R_e\tilde \nu_e-a_e\cdot \tilde \nu_e\tilde \nu_e-R_e\tilde\nu_e)\cdot (a_e+R_e\tilde \nu_e)
		=\tilde r^{-2}(a_e\cdot a_e-(a_e\cdot\tilde \nu_e)^2).
		$$
		As this term can be estimated by $c\tau_e^2$ the first claim follows. The second claim is a straight-forward application of Lemma \ref{muldelel} and Lemma \ref{ more comp}.
	\end{proof}
\end{lem}

 The next lemma provides some basic control on the evolution of many geometric quantities. However, it will turn out later on that these estimates can be sharpened in a considerable way.
\begin{lem}
	Let $\Sigma_t$ be an admissible area preserving Willmore flow and let $R_0^{-1}<\epsilon$. The following estimates hold:
	\label{good estimates}
	\begin{align*}
	|\Acirc|^4_{L^\infty(\Sigma_t)} &\leq c|\Acirc|^2_{L^2(\Sigma_t)}(R_e^{-6} - \partial_t \mathcal{W}(\Sigma_t)),
	\\
	|\Acirc|^2_{L^2(\Sigma_t)}&\leq c R_e^{-2}-cR_e^{4} \partial_t\mathcal{W}(\Sigma_t),
	\\
	|\nabla H|_{L^2(\Sigma_t)}^4 &\leq c|\Acirc|_{L^2(\Sigma_t)}(R_e^{-6}-\partial_t \mathcal{W}(\Sigma_t)),
	\\|\Delta H|_{L^2(\Sigma_t)}^2 &\leq cR_e^{-6}-c \partial_t \mathcal{W}(\Sigma_t),
	\\|W|_{L^2(\Sigma_t)}^2 &\leq cR_e^{-6}-c \partial_t \mathcal{W}(\Sigma_t),
	\\ \lambda ^2 &\leq cR_e^{-6}-c\epsilon\partial_t \mathcal{W}(\Sigma_t),
	\\|\nabla^2 H|_{L^2(\Sigma_t)}^4 &\leq cR_e^{-6}-c \partial_t \mathcal{W}(\Sigma_t),
	\\
	|H-\overline{H}_S|^4_{L^\infty(\Sigma_t)} &\leq (|\Acirc|^2_{L^2(\Sigma_t)}+cR_e^{-2})(R_e^{-6}-\partial_t\mathcal{W}(\Sigma_t)).
	\end{align*}
\end{lem}
\begin{proof}
	Lemma \ref{linfty} implies that
	\begin{align}
	|\Acirc|_{L^\infty(\Sigma_t)}^4\leq c|\Acirc|_{L^2(\Sigma_t)}^2(|W|^2_{L^2(\Sigma_t)}+R_e^{-6})
	\leq c\epsilon(|W|^2_{L^2(\Sigma_t)}+R_e^{-6}).
	\label{binftyest}
	\end{align}
	Next, using (\ref{h l2 est}) and $R_0^{-1}\leq  \epsilon$ we obtain
	$$
	|H-\overline{ H_S}|_{L^2(\Sigma_t)}\leq c(\eta R_e^{-2}+\tau_em R_e^{-1}+c|\Acirc_e|_{L^2})\leq c\sqrt{\epsilon}.
	$$
Integrating by parts and using Young's inequality we conclude that
	\begin{align}
	\int_{\Sigma_t} |\nabla H|^2d\mu=\int_{\Sigma_t} (\overline{H}_S-H)\Delta Hd\mu\leq c\sqrt{\epsilon}|\Delta H|_{L^2(\Sigma_t)}. \label{grad h l2 est}
	\end{align}
	Now we use these estimates to obtain an estimate for $|\Delta H|_{L^2(\Sigma_t)}$. Recall that (c.f. (\ref{h2 growth}))
	\begin{align}
	2\partial_t \mathcal{W}(\Sigma_t)=-|W|^2_{L^2}(\Sigma_t)+\lambda \int_{\Sigma_t} HWd\mu=-|W|^2_{L^2(\Sigma_t)}+|H|_{L^2(\Sigma_t)}^{-2}\bigg(\int_{\Sigma_t} WHd\mu\bigg)^2.
	\label{delta h first}
	\end{align}
	Integrating by parts, using (\ref{grad h l2 est}), (\ref{l2 curvature bound}), $|\operatorname{Rc}|\leq cR_e^{-3}$ and (\ref{binftyest}), we obtain 
	\begin{align}
	\bigg(\int_{\Sigma_t} WHd\mu\bigg)^2=& \bigg(\int_{\Sigma_t} \bigg(|\nabla H|^2-\overline{Rc}(\nu,\nu)H^2-|\Acirc|^2H^2\bigg)d\mu \bigg)^2 \notag
	\\\leq& c\epsilon|\Delta H|^2_{L^2(\Sigma_t)}+c\epsilon|W|_{L^2(\Sigma_t)}^2	+cR_e^{-6}.
	\label{delta h second}
	\end{align}
	Again by (\ref{binftyest}), $|\operatorname{Rc}|\leq cR_e^{-3}$  and the definition of $W$ we find
	\begin{align}
	|W|_{L^2(\Sigma_t)}^2 \geq |\Delta H|^2_{L^2(\Sigma_t)} -cR_e^{-6}-c\epsilon|W|_{L^2(\Sigma_t)}^2 \label{delta h third}.
	\end{align}
Combining (\ref{delta h first}), (\ref{delta h second}) and (\ref{delta h third}) we find
	\begin{align}
	|\Delta H|^2_{L^2(\Sigma_t)}\leq -c\partial_t\mathcal{W}(\Sigma_t)+cR_e^{-6},
	\label{delta proforma}
	\end{align}
	provided $\epsilon$ is sufficiently small.
	Returning to (\ref{delta h first}) and (\ref{delta h second}) we conclude
	$$
	|W|_{L^2(\Sigma_t)}^2 \leq -c\partial_t\mathcal{W}(\Sigma_t)+cR_e^{-6}.
	$$
	This implies the third, fourth and fifth estimate. Returning to (\ref{binftyest}) then implies the first estimate. Now we can use (\ref{lagrange}), (\ref{l2 curvature bound}), $|\operatorname{Rc}|\leq cR_e^{-3}$ and the first and third estimate to conclude that
	$$
	\lambda^2 \leq c(|\nabla H|^4_{L^2}+|H|\Acirc||_{L^2}^4+\bigg(\int_\Sigma H^2\overline{\operatorname{Rc}}(\nu,\nu)d\mu\bigg)^2)\leq -c\epsilon\partial_t\mathcal{W}(\Sigma_t)+cR_e^{-6}.
	$$
	Next, in the situation of Lemma \ref{improved BL2} we apply (\ref{rct est}) and the estimate $$
	|\nabla H|^2_{L^2(\Sigma_t)}\leq |H-\overline{ H_S}|_{L^2(\Sigma_t)}|\Delta H|_{L^2(\Sigma_t)} \leq \kappa R_e^{-2} |H-\overline{ H_S}|^2_{L^2(\Sigma_t)}+C(\kappa)R_e^4|\Delta H|^2_{L^2(\Sigma_t)},
	$$ valid for any $\kappa>0$, together with (\ref{h l2 est}). Absorbing the $|\Acirc|_{L^2(\Sigma_t)}^2$ terms and using the estimate (\ref{delta proforma}) then implies the second estimate. The two missing estimates are now straight-forward consequences of Lemma \ref{bochner} and Lemma \ref{H linfty estimate}.

\end{proof}
As promised, we now prove the refined a-priori estimates.
\begin{lem}
	If $\epsilon, R_0^{-1}$ are chosen sufficiently small, then an admissible area preserving Willmore flow satisfies the improved estimates
	\label{finalest}
\begin{align*}
|\Delta H|_{L^2(\Sigma_t)}^2 &\leq -c\partial_t\mathcal{W}(\Sigma_t)+c(\tau_e^2+R_e^{-1})R_e^{-6},\\
|\nabla H|_{L^2(\Sigma_t)}^4 &\leq -c\epsilon\partial_t\mathcal{W}(\Sigma_t)+c(\tau_e^2+R_e^{-2}) R_e^{-8},
\\
 |\Acirc|_{L^2(\Sigma_t)}^2 &\leq-cR_e^{4}\partial_t\mathcal{W}(\Sigma_t)+ c(\tau_e^2+R_e^{-1})R_e^{-2},
\\
(\lambda-2mR_e^{-3})^2&\leq-c\epsilon\partial_t\mathcal{W}(\Sigma_t)+ c(\tau_e^2+R_e^{-1})R_e^{-6}.
\end{align*}
\end{lem}
\begin{proof}
	First, let us recall that the asymptotic behaviour of the metric implies that $|\operatorname{Rc}|\leq c R_e^{-3}$ as well as $|\operatorname{Rc}-\operatorname{Rc}_S|\leq cR_e^{-4}$. Moreover, according to Lemma \ref{identities} there holds $|\nu-\nu_S|\leq cR_e^{-2}$ as well as $|d\mu-d\mu_S|\leq cR_e^{-4}$ while (\ref{l2 curvature bound}) states that $|H|_{L^2(\Sigma_t)}\leq c$. We will use these estimates at various points without explicitly stating them. There holds
	\begin{align}
	\notag-2\partial_t\mathcal{W}(\Sigma_t)=&|W|^2_{L^2(\Sigma_t)}-\lambda^2|H|^2_{L^2(\Sigma_t)} \\\notag
	=& |\Delta H|^2_{L^2(\Sigma_t)}+2|H\operatorname{Rc}(\nu,\nu)|^2_{L^2(\Sigma_t)}
		+2||\Acirc|^2H|^2_{L^2(\Sigma_t)}\\\notag&+2\langle \Delta H H,\operatorname{Rc}(\nu,\nu)\rangle_{L^2(\Sigma_t)}+2\langle \Delta H H,|\Acirc|^2\rangle_{L^2(\Sigma_t)}+2\langle H^2|\Acirc|^2,\operatorname{Rc}(\nu,\nu)\rangle_{L^2(\Sigma_t)} \\\notag
		&-|H|^{-2}_{L^2(\Sigma_t)}(|\nabla H|_{L^2(\Sigma_t)}^4+\langle H^2,\operatorname{Rc(\nu,\nu)}\rangle^2_{L^2(\Sigma_t)}+|H|\Acirc||_{L^2(\Sigma_t)}^4) \\\notag
		&+2|H|^{-2}_{L^2(\Sigma_t)}(|\nabla H|^2_{L^2(\Sigma_t)} \langle H^2, \operatorname{Rc}(\nu,\nu)\rangle_{L^2(\Sigma_t)}+|\nabla H|^2_{L^2(\Sigma_t)}|\Acirc H|_{L^2(\Sigma_t)}^2
		\\&-|\Acirc H|_{L^2(\Sigma_t)}^2\langle H^2, \operatorname{Rc}(\nu,\nu)\rangle_{L^2(\Sigma_t)}). \label{0eq}
	\end{align}
	We denote the 12 terms in the last equation by the Latin numbers $I-XII$. There holds
	\begin{align}
	|III+V+IX+XI|&\leq c |\Acirc|_{L^\infty}^4+c|\Acirc|_{L^\infty(\Sigma_t)}^2(|\Delta H|^2_{L^2(\Sigma_t)}+|\nabla H|^2_{L^2(\Sigma_t)}) \notag
	\\&\leq c|\Acirc|_{L^\infty(\Sigma_t)}^4+\frac{1}{8}|\Delta H|_{L^2(\Sigma_t)}^2+c|\nabla H|_{L^2(\Sigma_t)}^4 \notag 
	\\& \leq c|\Acirc|^2_{L^2(\Sigma_t)}(R_e^{-6}-\partial_t\mathcal{W}(\Sigma_t))+\frac{1}{8}|\Delta H|_{L^2(\Sigma_t)}^2+c|\nabla H|_{L^2(\Sigma_t)}^4,
	\label{1eq}
	\end{align}
	where we used Lemma \ref{good estimates} in the last inequality. We now focus on $IV$ and $X$. We replace  $\operatorname{Rc}$ by $\operatorname{Rc}_S$, $\nu$ by $\nu_S$ and $d\mu$ by $d\mu_S$. This results in error terms that can be estimated by
	\begin{align}
	c\eta R_e^{-4}(|\Delta H|_{L^2(\Sigma_t)}+|\nabla H|^2_{L^2(\Sigma_t)}) \leq c  R_e^{-8}+\frac18|\Delta H|_{L^2(\Sigma_t)}^2.
	\label{2eq}
	\end{align}
	In the last inequality, we used the crude estimate $|\nabla H|_{L^2(\Sigma_t)}^2\leq c|\Delta H|_{L^2(\Sigma_t)}$, see (\ref{grad h l2 est}). In order to estimate these terms further, we express them in terms of the approximate sphere $S$ from Lemma \ref{muldelel}. To this end, we denote the conformal parametrization $S\to\Sigma_t$ by $\psi$ and indicate the respective geometric quantities of $S$ by a tilde. According to Lemma \ref{muldelel} we have
	$|r^{-3}-R_e^{-3}|\leq R_e^{-3}(|\Acirc_e|_{L^2(\Sigma_t)}+\tau_e)$ as well as $|\partial_r-\tilde \partial_r|\leq |\Acirc_e|_{L^2(\Sigma_t)}$ while Lemma \ref{r compare} and Taylor's theorem imply that $|\phi^{-6}-1|\leq cR_e^{-1}$. Since $|\tilde \nu_e-\nu_e|\leq c|\nabla_e \psi- \tilde \nabla_e\operatorname{Id}|$ we find
	\begin{align*}
	|\operatorname{Rc}_S(\nu_S,\nu_S)d\mu_S-mR_e^{-3}(1-3(\tilde \partial_r\cdot \tilde \nu_e)^2)d \mu_e|\leq cR_e^{-3}(\tau_e+|\Acirc_e|_{L^2(\Sigma_t)})+cR_e^{-3}|\nabla_e \psi-\nabla_e \operatorname{id}|.
	\end{align*}
	Consequently, it follows from H\"older's inequality and Lemma \ref{muldelel} that replacing $\operatorname{Rc}_S(\nu_S,\nu_S)d\mu_S$ by $mR_e^{-3}(1-3(\tilde \partial_r\cdot \tilde \nu_e)^2)d\mu_e$ in $IV$ and $X$ results in error terms that can be estimated by 
	\begin{align}
	&cR_e^{-3}(|\Delta HH|_{L^2(\Sigma_t)}R_e|\Acirc_e|_{L^2(\Sigma_t)}+(|\Delta H|_{L^2(\Sigma_t)}+|\nabla H|^2_{L^2(\Sigma_t)})
	(\tau_e+|\Acirc_e|_{L^2(\Sigma_t)})\notag \\\notag &+|\nabla  H|^2_{L^2(\Sigma_t)}|H^2|_{L^2(\Sigma_t)}R_e|\Acirc_e|_{L^2(\Sigma_t)})
	\\ \leq &c\tau_e^2 R_e^{-6}+c|\Acirc_e|^2_{L^2(\Sigma_t)}R_e^{-6}+1/8 |\Delta H|^2_{L^2(\Sigma_t)}+c|\nabla H|^4_{L^2(\Sigma_t)}+cR_e^{-4}|\Acirc_e|^2_{L^2(\Sigma_t)}(R_e^{-2}+R_e^2|H|^4_{L^\infty})
	\notag \\ \leq &c\tau_e^2R_e^{-6}+c|\Acirc|^2_{L^2(\Sigma_t)} R_e^{-6}+cR_e^{-8}-c\epsilon \partial_t\mathcal{W}(\Sigma_t)+1/8 |\Delta H|^2_{L^2(\Sigma_t)}+c|\nabla H|^4_{L^2(\Sigma_t)}.
	\label{3eq}
	\end{align}
	In the first inequality we used Young's inequality, H\"older's inequality and the estimate $|H^2|^2_{L^2(\Sigma_t)}\leq cR_e^2|H|_{L^\infty(\Sigma_t)}^4$. In the second inequality we used the estimate $|H|^4_{L^\infty(\Sigma_t)}\leq |H-\overline{ H_S}|^4_{L^\infty(\Sigma_t)}+cR_e^{-4}$, Lemma \ref{good estimates} to estimate
	$$
	|H-\overline{ H_S}|^4_{L^\infty(\Sigma_t)}\leq cR_e^{-6}|\Acirc|^2_{L^2(\Sigma_t)}+cR_e^{-8}-c\epsilon\partial_t\mathcal{W}(\Sigma_t)	$$
	 as well as Lemma \ref{identities} to replace $|\Acirc_e|_{L^2(\Sigma_t)}$ by $|\Acirc|_{L^2(\Sigma_t)}$. 	 
	  Performing the same two procedures  with $VI$ yields an error term that can be estimated by
\begin{align}
&cR_e^{-2}|H^2\Acirc^2|_{L^2(\Sigma_t)}(|\Acirc|_{L^2(\Sigma_t)}+\tau_e+R_e^{-1})\notag \\\notag 
\leq & c(R_e^{-1}|H-\overline{ H_S}|^2_{L^\infty(\Sigma_t)}|\Acirc|^2_{L^\infty(\Sigma_t)}+R_e^{-3}|\Acirc|^2_{L^\infty(\Sigma_t)})
(|\Acirc|_{L^2(\Sigma_t)}+\tau_e+R_e^{-1}) \\ \notag 
\leq & c(|\Acirc|^4_{L^\infty(\Sigma_t)}+R_e^{-2}|H-\overline{ H_S}|^4_{L^\infty(\Sigma_t)})+cR_e^{-3}|\Acirc|^2_{L^\infty(\Sigma_t)}(|\Acirc|_{L^2(\Sigma_t)}+\tau_e+R_e^{-1})
\\\leq & cR_e^{-6}|\Acirc|_{L^2(\Sigma_t)}^2+cR_e^{-6}(\tau_e^2+R_e^{-2})-c\epsilon\partial_t\mathcal{W}(\Sigma_t).
\label{4eq}
\end{align}
In the first inequality we used Young's inequality as well as the estimate $|H|_{L^\infty(\Sigma_t)}^2\leq c |H-\overline{H_S}|_{L^\infty(\Sigma_t)}^2+cR_e^{-2}$. In the second inequality we used Young's inequality again and the fact that  $|\Acirc|_{L^2(\Sigma_t)}+\tau_e+R_e^{-1}$ is bounded. In the third inequality we used Young's inequality one more time and Lemma \ref{good estimates} to estimate $|\Acirc|^4_{L^\infty(\Sigma_t)}$ and $|H-\overline{ H_S}|^4_{L^\infty(\Sigma_t)}$. Performing this procedure on $XII$ yields a similar error term.
	On the other hand, there holds
$	\tilde \partial_r \cdot \tilde \nu_e = r^{-1}(R_e\tilde \nu_e +a_e)\cdot \tilde \nu_e.$
	Again, $\tilde r^{-1}$ can be replaced by $R_e^{-1}$ and since $\tau_e=|a_e|/R_e$, we find that 
	\begin{align}
		|mR_e^{-3}(1-3(\tilde \partial_r\cdot\tilde\nu_e)^2)+2mR_e^{-3}|\leq c{\tau_eR_e^{-3}}. 
	\label{5eq}
	\end{align}
	Integrating by parts, we find that $$-2mR_e^{-3}\langle \Delta H H,1\rangle_{L^2(\Sigma_t)}-2mR_e^{-3}|H|^{-2}_{L^2(\Sigma_t)}|\nabla H|^2_{L^2(\Sigma_t)}\langle H^2,1\rangle_{L^2(\Sigma_t)}=0.$$
	Combining this with (\ref{2eq}), (\ref{3eq}) and (\ref{5eq})  we obtain
\begin{align}
|IV+X|&\leq c\tau_eR_e^{-3}|\Delta H|_{L^2(\Sigma_t)}+c\tau_e^2R_e^{-6}+cR_e^{-8}+c|\Acirc|^2_{L^2(\Sigma_t)} R_e^{-6}-c\epsilon \partial_t\mathcal{W}(\Sigma_t)\notag\\\notag&+1/8 |\Delta H|^2_{L^2(\Sigma_t)}+c|\nabla H|^4_{L^2(\Sigma_t)}
\\&\leq c\tau_e^2R_e^{-6}+cR_e^{-8}+c|\Acirc|^2_{L^2(\Sigma_t)} R_e^{-6}-c\epsilon \partial_t\mathcal{W}(\Sigma_t)+1/4 |\Delta H|^2_{L^2(\Sigma_t)}+c|\nabla H|^4_{L^2(\Sigma_t)}.
\label{6eq}
\end{align}
In a similar way we can use (\ref{4eq}) to find
\begin{align}
|VI+XII|&\leq c\tau_e R_e^{-3} |H\Acirc|_{L^2(\Sigma_t)}^2+cR_e^{-6}|\Acirc|_{L^2(\Sigma_t)}^2+cR_e^{-6}(\tau_e^2+R_e^{-2})-c\epsilon\partial_t\mathcal{W}(\Sigma_t) \notag
\\ &\leq \label{7eq} cR_e^{-6}|\Acirc|_{L^2(\Sigma_t)}^2+cR_e^{-6}(\tau_e^2+R_e^{-2})-c\epsilon\partial_t\mathcal{W}(\Sigma_t).
\end{align}
Here, we estimated $\tau_eR_e^{-3}|H\Acirc|^2_{L^2(\Sigma_t)}\leq c\tau_e^2R_e^{-6}+c|\Acirc|^4_{L^\infty(\Sigma_t)}$ and then used Lemma \ref{good estimates}.
In $II$ and $VIII$ we first replace every $H$ by $\overline{ H_S}$. In light of Lemma \ref{good estimates}, the error can be estimated by
$$
c|H-\overline{ H_S}|_{L^\infty(\Sigma_t)}R_e^{-5} \leq
c|H-\overline{ H_S}|^4_{L^\infty(\Sigma_t)}R_e^{-2}+cR_e^{-8}\leq cR_e^{-6}|\Acirc|_{L^2(\Sigma_t)}^2+cR_e^{-8}-c\epsilon\partial_t\mathcal{W}(\Sigma_t).
$$
From this we find that
\begin{align}
&|II+VIII| \notag \\\leq& cR_e^{-2} \bigg|\int_{\Sigma_t}(\operatorname{Rc}(\nu,\nu))^2d\mu
-|\Sigma_t|^{-1}\bigg(\int_{\Sigma_t} \operatorname{Rc}(\nu,\nu)d\mu\bigg)^2
\bigg|+ c|\Acirc|^2_{L^2(\Sigma_t)}R_e^{-6}+cR_e^{-8}-c\epsilon\partial_t\mathcal{W}(\Sigma_t)
\notag\\\notag\leq& c|\nabla \operatorname{Rc}(\nu,\nu)|^2_{L^2(\Sigma_t)}+ c|\Acirc|^2_{L^2(\Sigma_t)}R_e^{-6}+c R_e^{-8}-c\epsilon\partial_t\mathcal{W}(\Sigma_t)
\\\notag \leq& c|\overline\nabla \operatorname{Rc}(\nu,\nu)|^2_{L^2(\Sigma_t)}+c|A|^2_{L^\infty(\Sigma_t)}|\operatorname{Rc}(\nu,\cdot)^T|_{L^2(\Sigma_t)}^2+ c|\Acirc|^2_{L^2(\Sigma_t)}R_e^{-6}+c R_e^{-8}-c\epsilon\partial_t\mathcal{W}(\Sigma_t) 
\\\leq& c|\Acirc|^2_{L^2(\Sigma_t)}R_e^{-5}+c\tau_e^2R_e^{-6}+c R_e^{-7}-c\epsilon\partial_t\mathcal{W}(\Sigma_t).
\label{8eq}
\end{align}
In the second inequality, we used the Poincare inequality with zero mean. In the second inequality we expressed $\nabla$ in terms of $\overline\nabla$ and $A$, see (\ref{ambient vs intrinsic derivative}), In the fourth inequality, we used (\ref{rct est}) and the estimate 
$|(\overline\nabla \operatorname{Rc})(\nu,\nu)|^2_{L^2(\Sigma_t)} \leq  c|\Acirc|^2_{L^2(\Sigma_t)} R_e^{-5}+c\tau_e^2R_e^{-6}+cR_e^{-7}$, which can be shown in the same fashion as (\ref{rct est}). Finally, we used $|A|^2_{L^\infty(\Sigma_t)}\leq 2|\Acirc|^2_{L^\infty(\Sigma_t)}+
2|H|^2_{L^\infty(\Sigma_t)}$ and estimated these terms in the usual way. Combining (\ref{0eq}),(\ref{1eq}),(\ref{6eq}),(\ref{7eq}) and (\ref{8eq}) we finally obtain 
\begin{align}
|\Delta H|_{L^2(\Sigma_t)}^2 \leq -c\partial_t\mathcal{W}(\Sigma_t)+ c|\nabla H|_{L^2(\Sigma_t)}^4+c(\tau^2_e+ R_e^{-1})R_e^{-6}+cR_e^{-5}|\Acirc|_{L^2(\Sigma_t)}^2.
\label{bestest}
\end{align}
In light of the inequality  \begin{align}|\nabla H|_{L^2(\Sigma_t)}^4 \leq |H-\overline {H_S}|_{L^2(\Sigma_t)}^2|\Delta H|_{L^2(\Sigma_t)}^2\leq c(\tau_e^2R_e^{-2}+|\Acirc|_{L^2(\Sigma_t)}^2+CR_e^{-4})|\Delta H|^2_{L^2(\Sigma_t)}, \label{9eq}\end{align}
which follows from (\ref{h l2 est}) and the divergence theorem,
we can eventually absorb the second term on the right hand side in (\ref{bestest}) to obtain
\begin{align}
|\Delta H|_{L^2(\Sigma_t)}^2 \leq -c\partial_t\mathcal{W}(\Sigma_t)+c(\tau^2_e+ R_e^{-1})R_e^{-6}+cR_e^{-5}|\Acirc|_{L^2(\Sigma_t)}^2.
\label{verybestest}
\end{align}
 Next, using Lemma \ref{improved BL2}, (\ref{rct est}) and  Young's inequality we find for any $\kappa>0$
\begin{align*}
|\Acirc|^2_{L^2(\Sigma_t)}\leq& c|\Acirc|^2_{L^2(\Sigma_t)}R_e^{-1}+c\tau_e^2 R_e^{-2}+cR_e^{-3}+\kappa |H-\overline{ H_S}|^2_{L^2(\Sigma_t)} \\&+C(\kappa)R_e^{4}|H-\overline{ H_S}|^{-2}_{L^2(\Sigma_t)}|\nabla H|^4_{L^2(\Sigma_t)}
\\ \leq& |\Acirc|^2_{L^2(\Sigma_t)}(\kappa+cR_e^{-1})+c\tau_e^2 R_e^{-2}+cR_e^{-3}-cR_e^{4}\partial_t\mathcal{W}.
\end{align*}
In the last step, we have used (\ref{h l2 est}), 
$|\nabla H|^4_{L^2(\Sigma_t)}\leq |H-\overline {H_S}|_{L^2(\Sigma_t)}^2|\Delta H|_{L^2(\Sigma_t)}^2$ and (\ref{verybestest}). Absorbing yields the claimed estimate for $|\Acirc|^2_{L^2(\Sigma_t)}$. Reinserting into (\ref{verybestest}) gives the claimed estimate for $|\Delta H|^2_{L^2(\Sigma_t)}$. This then implies
\begin{align*}
|\nabla H|_{L^2(\Sigma_t)}^4\leq &|H-\overline {H_S}|_{L^2(\Sigma_t)}^2|\Delta H|_{L^2(\Sigma_t)}^2\\
&\leq c(\tau_e^2R_e^{-2} +|\Acirc|^2_{L^2(\Sigma_t)}+cR_e^{-4})(-\partial_t\mathcal{W}(\Sigma_t)+\tau_e^2 R_e^{-6}+cR_e^{-7})\\
&\leq -c\epsilon \partial_t\mathcal{W}+c\tau_e^2 R_e^{-8}+cR_e^{-10},
\end{align*}
as claimed. Finally, we recall the definition of $\lambda$, see (\ref{lagrange}). We have 
$$
\lambda \leq c|\nabla H|_{L^2(\Sigma_t)}^2+c|\Acirc|_{L^\infty(\Sigma_t)}^2-|H|_{L^2(\Sigma_t)}^{-2}\int_{\Sigma_t} H^2 \operatorname{Rc(\nu,\nu)}d\mu.
$$
Using the same methods as before, the last term can be computed explicitly to give
$$
(\lambda-2mR_e^{-3})^2\leq c|H-\overline{ H_S}|^2_{L^\infty(\Sigma_t)}R_e^{-4}+c\tau_e^2R_e^{-6}+c R_e^{-8}+c|\Acirc|^2_{L^2(\Sigma_t)}R_e^{-4}+c|\Acirc|^4_{L^\infty(\Sigma_t)}+|\nabla H|^4_{L^2(\Sigma_t)}.
$$
The claim follows.
	\end{proof}

\section{The evolution of the barycenter}

In this section, we proof Theorem \ref{main theorem2} and Theorem \ref{main theorem1}. To this end, we derive a differential inequality for $\tau_g$. In the next Lemma, we show that the evolution of $\tau_g$ is linked to the translation sensitivity of $\mathcal{W}$ in terms of the Schwarzschild background. We note that the symbol $"\cdot"$ indicates the Euclidean inner product.
\begin{lem}
	Let $\Sigma_t$ be an admissible area preserving Willmore flow. Then the following holds
	\begin{align}
	\notag \partial_t \tau_g=&\frac{1}{|\Sigma_t|_gR_g(t)}b_g\cdot\bigg(\int_{\Sigma_t} \nu(\Delta H +H{\operatorname{Rc}}(\nu,\nu)+H|\Acirc|^2+\lambda H)d\mu\\&+\int_{\Sigma_t} (x-a_e)(H\Delta H+H^2{\operatorname{Rc}}(\nu,\nu)+H^2|\Acirc|^2+\lambda H^2)d\mu\bigg), \label{second line}
	\end{align}
	where $b_g=a_g/|a_g|$.
	Moreover, we have
	\begin{align}
	\partial_t \tau_g=&\frac{3-2mR_g^{-1}}{|\Sigma_t|_gR_g(t)}\bigg(b_g\cdot\int_{\Sigma_t} \nu_S(\Delta_S H_S \label{first aprox} +H_S{\operatorname{Rc_S}}(\nu_S,\nu_S)+H_S|\Acirc_S|^2+\lambda H_S)d\mu_S
	\\&+\mathcal{O}((\tau^2_e+\eta+R_e^{-1})R_e^{-3}-R_e^3\partial_t\mathcal{W}(\Sigma_t))
	\bigg). \notag
	\end{align}
	
\end{lem}
\begin{proof}
The first identity is a straight forward computation using the flow equation (\ref{flow equation}) and the fact that $HW+\lambda H^2$ has zero mean. We first show that the second line of (\ref{second line}) is $2(1-m/R_g)$ times  the first line of (\ref{second line}) up to an error term by replacing $(x-a_e)H$ by $(2(1-m/R_g))\nu$. To this end, we first replace $H$ by $\overline{ H_S}$ in the second line of (\ref{second line}) which according to (\ref{h l2 est}) yields an error term that can be estimated by
\begin{align}
cR_e(|\Acirc|_{L^2(\Sigma_t)}+\tau_eR_e^{-1}+\eta R_e^{-2})|W+\lambda H|_{L^2(\Sigma_t)}.
\label{eq1}
\end{align}
Similarly, Lemma \ref{muldelel} and Lemma \ref{identities} imply that replacing $(x-a_e)$ by $(\tilde x-a_e)=R_e\tilde\nu_e$ and then $\tilde \nu_e$ by $\nu_e$ results in error terms that can be estimated by
\begin{align}
cR_e(|\Acirc|_{L^2(\Sigma_t)}+\eta R_e^{-2})|W+\lambda H|_{L^2(\Sigma_t)}.
\label{eq2}
\end{align}
Finally, replacing $R_e\overline{ H_S}=2\phi^{-2}(R_e)-2m R_e^{-1}\phi^{-3}(R_e)$ by $2\phi^{-2}(r)(1-m/R_g)$ and then $\nu_S=\phi^{-2}\nu_e$ by $\nu$ leads to an error that can be estimated by
\begin{align}
cR_e(|\Acirc|_{L^2(\Sigma_t)}+\tau_e R_e^{-1}+\nu R_e^{-2})|W+\lambda H|_{L^2(\Sigma_t)}.
\label{eq3}
\end{align}
Here, we also used Lemma \ref{ more comp}. We observe that
\begin{align}
|W+\lambda H|_{L^2(\Sigma_t)}^2=|W|_{L^2(\Sigma_t)}^2-\lambda \int_{\Sigma_t}WHd\mu=-\partial_t\mathcal{W}(\Sigma_t).
\label{trick}
\end{align}
Therefore, we obtain
\begin{align}
(c\tau_e+cR_e^{-1})|W+\lambda H|_{L^2(\Sigma_t)}\leq c(\tau_e^2+R_e^{-2})R_e^{-3}-cR_e^{3}\partial_t\mathcal{W}
\label{eq4}
\end{align}
Moreover, using Lemma \ref{finalest} and Young's inequality, we estimate
\begin{align}
cR_e|\Acirc|_{L^2(\Sigma_t)}|W+\lambda H|_{L^2(\Sigma_t)} \leq c(\tau_e^2+R_e^{-1})R_e^{-3}-cR_e^{3}\partial_t\mathcal{W}.
\label{eq5}
\end{align}
Combining (\ref{eq1})-(\ref{eq5}) shows that
 \begin{align*}
 \partial_t \tau_g=&\frac{3-2mR_g^{-1}}{|\Sigma_t|_gR_g(t)}\bigg(b_g\cdot\int_{\Sigma_t} \nu(\Delta H +H{\operatorname{Rc}}(\nu,\nu)+H|\Acirc|^2+\lambda H)d\mu
 \\&+\mathcal{O}(\tau^2_eR_e^{-3}+R_e^{-4}-R_e^3\partial_t\mathcal{W}(\Sigma_t))
 \bigg).
 \end{align*}
Next, we would like to express this quantity in terms of the Schwarzschild geometry. To this end, we will make implicit use of Lemma \ref{identities} and Lemma \ref{finalest}.
	Recalling  $|d\mu-d\mu_S|\leq c\eta R_e^{-2} d\mu $, $|\nu-\nu_S|\leq c\eta R_e^{-2}$ as well as $|H-H_S|\leq c\eta R_e^{-3}$ we find
\begin{align}
	\lambda \int_{\Sigma_t}\nu Hd\mu=\lambda \int_{\Sigma_t}\nu_S H_Sd\mu_S+\lambda\mathcal{O}(\eta R_e^{-1}).
	\label{err1a}
	\end{align}
	Using Lemma \ref{finalest}, this error can be estimated via 
	\begin{align}
	\eta\lambda R_e^{-1}\leq 2\eta^2 R_e^{-4}+2\lambda^2R_e^2 \leq cR_e^{-4}-c R_e^{2}\partial_t \mathcal{W}(\Sigma_t). \label{err1b}
	\end{align}
	Next, we have
	\begin{align*}
	\int_{\Sigma_t}\nu H|\Acirc|^2d\mu=\int_{\Sigma_t} \nu_S H_S|\Acirc|^2d\mu_S+\mathcal{O}(\eta|\Acirc|^2_{L^\infty(\Sigma_t)}R_e^{-2}).
	\end{align*}
	According to Lemma \ref{identities} there holds $||\Acirc|^2-|\Acirc_S|^2|\leq c\eta(|\Acirc|+R_e^{-3}+|A|R_e^{-2})(R_e^{-3}+|A|R_e^{-2})$. From this it follows that replacing $|\Acirc|^2$ by $|\Acirc_S|^2$ in the previous equation yields an error that can be estimated by
	\begin{align*}
	&c\eta \int_{\Sigma_t} |H_S|\bigg(|\Acirc|+R_e^{-3}+|A|R_e^{-2}\bigg)\bigg(R_e^{-3}+|A|R_e^{-2}\bigg)d\mu 
	\notag \\\leq&c\eta \notag (R_e^{-4}+|\Acirc|^2_{L^2(\Sigma_t)}R_e^{-2}+|\Acirc|_{L^\infty(\Sigma_t)}|H|_{L^\infty(\Sigma_t)}R_e^{-1}+|H|^2_{L^\infty(\Sigma_t)}R_e^{-3})
	\\\leq&c\eta R_e^{-3}-cR_e^2\partial_t\mathcal{W}(\Sigma_t) 
	\end{align*}
	In the first inequality we used the crude estimate $|H_S|\leq cR_e^{-1}+|H|$, applied Young's inequality several times and used (\ref{l2 curvature bound}). In the second inequality we used Young's inequality, Lemma \ref{good estimates}  as well as  Lemma \ref{finalest}. Hence, 
	\begin{align}
	\int_{\Sigma_t} \nu H|\Acirc|^2d\mu=\int_{\Sigma_t} \nu_S H_S|\Acirc_S|^2d\mu_S
	+\mathcal{O}(\eta (R_e^{-3}-R_e^{2}\partial_t\mathcal{W}(\Sigma_t))). \label{err2}
	\end{align}
	Using the asymptotic behaviour of ${\operatorname{Rc}},\nu,d\mu$ it is easy to see that 
	\begin{align}
	\int_{\Sigma_t} \nu H{\operatorname{Rc}}(\nu,\nu)d\mu
	= 
	\int_{\Sigma_t} \nu_S H_S{\operatorname{Rc}}_S(\nu_S,\nu_S)d\mu_S + \mathcal{O}(\eta R_e^{-3}).
	\label{err3}
	\end{align}
	Regarding the last term, we note
	$$
	\int \nu \Delta Hd\mu =\int \nu_S \Delta Hd\mu_S +\mathcal{O}(\eta R_e^{-3}+\eta R_e|\Delta H|^2_{L^2}).
	$$
	It is then straightforward to see that any smooth function $f$ satisfies the estimate $|\Delta_S f-\Delta f|\leq c\eta (|\nabla^2 f|R_e^{-2}+|\nabla f|R_e^{-3})$. Hence replacing $\Delta$ by $\Delta_S$ results in an error that can be estimated by
	\begin{align*}
	c\eta\int_{\Sigma_t} (R_e^{-2}|\nabla ^2H|+R_e^{-3}|\nabla H|)d\mu_S
	& \leq c\eta R_e^{-3}+c\eta R_e|\nabla^2 H|^2_{L^2(\Sigma_t)}+c\eta R_e^{-1}|\nabla H|^2_{L^2(\Sigma_t)}  \\ &\leq c\eta R_e^{-3}
	+c\eta R_e|\nabla ^2H|^2_{L^2(\Sigma_t)}+c\eta R_e |\nabla H|^4_{L^2(\Sigma_t)}
	\\ &\leq c\eta R_e^{-3}
	-c\eta R_e\partial_t\mathcal{W}(\Sigma_t),
	\end{align*}
	where we used Lemma \ref{good estimates} and Lemma \ref{finalest}. 	Integrating by parts we obtain
	$$
	\int_{\Sigma_t} \nu \Delta Hd\mu = -\int_{\Sigma_t} \nabla_S \nu_S\nabla_S Hd\mu_S+\mathcal{O}(\eta R_e^{-3}-c\eta R_e\partial_t\mathcal{W}(\Sigma_t)).
	$$
	There holds $\nabla _S \nu_S\leq cR_e^{-2}+|A_S|\leq cR_e^{-2}+c|A|$. Hence, using Lemma \ref{identities} to replace $\nabla_S H$ by $\nabla_S H_S$ gives an error that can be estimated by
	\begin{align*}
	c\eta \int_{\Sigma_t} (R_e^{-2}+|A|)(|\nabla A|R_e^{-2}+|A|R_e^{-3}+R_e^{-4}) &\leq c\eta R_e^{-3}+c\eta R_e^{-1}|\nabla A|^2_{L^2(\Sigma_t)}\\&\leq c\eta R_e^{-3}+\eta R_e |\nabla A|^4_{L^2(\Sigma_t)}.
	\end{align*}
	According to Lemma \ref{section 4 lemma with assumptions} and Lemma \ref{finalest} there holds 
	\begin{align}
	|\nabla A|^4_{L^2(\Sigma_t)}\leq c\lambda ^2+cR_e^{-6}+R_e^2|\Acirc|^4_{L^\infty(\Sigma_t)}\leq -cR_e^{2}\partial_t\mathcal{W}(\Sigma_t)+cR_e^{-6}. 
	\end{align}
	We conclude
	\begin{align}
	\int_{\Sigma_t} \nu \Delta Hd\mu =\int_{\Sigma_t} \nu_S \Delta_S H_Sd\mu_s +\mathcal{O}(\eta R_e^{-3}-\eta R_e^{3}\partial_t{\mathcal{W}}(\Sigma_t)).
	\label{err4}
	\end{align}
	The claim then follows from (\ref{err1a}), (\ref{err1b}), (\ref{err2}), (\ref{err3}) and (\ref{err4}).
\end{proof}
At this point, the central observation lies in the fact that the first three terms in (\ref{first aprox}) are a multiple of the variation of the Schwarzschild Willmore energy along a translation. Through approximation by a sphere, we can therefore explicitly compute $\partial_t \tau_g$ up to an error. We denote the geometric quantities of $\Sigma$ as usual, the geometric quantities of $S:=S_{a_e}(R_e)$ are denoted using a tilde.
\begin{lem}
	Let $\Sigma_t$ be an admissible area preserving Willmore flow. If $\delta$ is chosen sufficiently small (depending on $m$), then there holds 
	\begin{align*}
	\partial_t \tau_g \leq -160 m^2R_g^{-6}\tau_g+c(\tau_g^2+\epsilon+\eta)R_g^{-6}-c\partial_t\mathcal{W}.
	\end{align*}

\end{lem}
\begin{proof}
	Let $\alpha:=\nu_S\cdot b_g$, where $b_g:=a_g/|a_g|$. Using (\ref{ev equations}), it it easy to see that the first three terms of (\ref{first aprox}) are four times the variation of the Willmore energy in the Schwarzschild space under a translation in direction $-b_g$. We therefore have
	\begin{align}
	b_g\cdot\int_{\Sigma_t} \nu_S(\Delta_S H_S +H_S{ \operatorname{Rc_S}}( \nu_S,\nu_S)+H_S|\Acirc_S|_S^2)d\mu_S=-\frac{d}{ds}\bigg|_{s=0}\int_{\Sigma_t} H_S^2d\mu_S.
	\label{trans0}
	\end{align}
	Here the variation is given by $s\mapsto F+sb_g$, where $F$ is the embedding of $\Sigma_t\to M$. With the help of Lemma \ref{identities} we compute
	$$
	\int_{\Sigma_t} H_S^2d\mu_S=\int_{\Sigma_t} H_e^2 d\mu_e-4m\int_{\Sigma_t}H_er^{-2}\phi^{-1}\partial_r\cdot \nu_ed\mu_e+4m^2\int r^{-4}\phi^{-2}(\partial_r\cdot \nu_e)^2d\mu_e.
	$$
	The first term on the right hand side is four times the Euclidean Willmore energy and therefore invariant under translations. In fact, as translations are isometries of $\mathbb{R}^3$, the quantities $H_e,d\mu_e,\nu_e$ are all invariant under translations. From this it is easy to see that 
	\begin{align}
	-\frac{d}{ds}\bigg|_{s=0}\int_{\Sigma_t} H_S^2 d\mu_S = -4m\int_{\Sigma_t} H_e \partial_s(r^{-2}\partial_r)|_{s=0}\cdot \nu_e d\mu_e+\mathcal{O}(R_e^{-3}).
	\label{trans1}
	\end{align}
	The terms summarized in $\mathcal{O}(R_e^{-3})$ can be expressed in terms of the approximating sphere using Lemma \ref{muldelel} and the fact that $|H_e|_{L^2(\Sigma_t)}^2$ is bounded which follows from Lemma \ref{identities} and (\ref{l2 curvature bound})\footnote{For instance, one has $|r^5-\tilde r^{5}|\leq c R_e^{-6}|r-\tilde r|\leq c R_e^{-5}|\Acirc_e|_{L^2(\Sigma_t)}$.}. Using Lemma \ref{identities} and Lemma \ref{finalest}, the resulting error can then be estimated by
	\begin{align}
	cR_e^{-3}|\Acirc_e|_{L^2(\Sigma_t)} \leq cR_e^{-4}-cR_e^{2}\partial_t\mathcal{W}(\Sigma_t).
	\label{trans2}
	\end{align}
	Consequently, we summarize such terms by the letter $\Gamma$ and focus on the only term which is of higher order and compute
	\begin{align}
	\notag&m\int_{\Sigma_t} H_e \partial_s(r^{-2}\partial_r)|_{s=0}\cdot \nu_e d\mu_e\\=\notag&m\int_{\Sigma_t} H_e r^{-3}(b_g-3\partial_r\cdot b_g\partial_r )\cdot \nu_e d\mu_e \\
	=&\notag\int_{\Sigma_t} H_e \operatorname{Rc}_S(\nu_S,b_g)d\mu_S+\Gamma \\
	=&\notag\int_{\Sigma_t} (H_e-\tilde H_e) \operatorname{Rc}_S(\nu_S,b_g)d\mu_S+\Gamma\\
	=\notag&m\int_{\Sigma_t} (H_e-\tilde H_e) r^{-3}(b_g\cdot \nu_e-3b_g\cdot \partial_r \partial_r\cdot \nu_e) d\mu_S+\Gamma
	\\
	=\notag&(-2mR_e^{-3})\int_{\Sigma_t} (H_e-\tilde H_e) \alpha d\mu_e+\mathcal{O}((\tau_e^2+R_e^{-1}+\eta)R_e^{-3}-R_e^{3}\partial_t\mathcal{W}(\Sigma_t))+\Gamma
	\\ \label{trans3}
	=&\mathcal{O}((\tau_e^2+R_e^{-1}+\eta)R_e^{-3}-R_e^{3}\partial_t\mathcal{W}(\Sigma_t))+\Gamma.
	\end{align} 
	In the second equation we used (\ref{rc schwarzschild}). In the third equation, we used the fact that $$\int_{\Sigma_t} \operatorname{Rc_S}(\nu_S,b_g)d\mu_S=0.$$ This follows from the so-called Pohozaev identity, see (5.12) in \cite{lamm2011foliations}: In Schwarzschild it states that for any vector field $X$ and any compact, null-homologous and smooth domain $\Omega$ the following relation holds
	\begin{align}
	\frac12\int_{\Omega} g_S(\mathcal{D}X, \operatorname{Rc}_S)d\mu_S=\int_{\partial\Omega}\operatorname{Rc}_S(X,\nu_S)d\mu_S,
	\label{pohozaev}
	\end{align}
	where $\mathcal{D}$ denotes the conformal killing operator\footnote{There is an easy way to see this using the potential function $f=(2-\phi)/\phi$ and the fact that $g_S$ is static, compare  (\ref{static equation}). In fact, there holds $\operatorname{div}_S\operatorname{Rc}_S(X,\cdot)=\operatorname{div}_S f^{-1}\nabla_S^2 f(X,\cdot)=\frac12g_S(\mathcal{D}X,f^{-1}\nabla^2_S f)-f^{-1}\nabla_S^*\nabla_S^2 f(X)-f^{-2}\nabla^2_Sf(X,\nabla_S f)$. Commuting derivatives and using $\Delta_S f=0$ we find $-\nabla_S^*\nabla^2_S f(X)=\operatorname{Rc}_S(x,\nabla_Sf)$. Using the static equation one more time it follows that $\operatorname{div}_S\operatorname{Ric}_S(X,\cdot)=\frac12g_S(\mathcal{D}X,f^{-1}\nabla^2_S f)$ and (\ref{pohozaev}) follows from the divergence theorem.  }. One then picks $X=b_g$ which is a conformal killing vector field and $\Omega$ to be the region bounded by $\Sigma_t$ and a coordinate sphere with radius tending to infinity. In the fourth equation, we again used (\ref{rc schwarzschild}) again. Then, in the fifth equation, we used Lemma \ref{muldelel} to replace $r$ by $R_e$, $\partial_r$ by $\tilde \partial_r$ and then $\tilde \partial_r=r^{-1}(\tilde \nu_e R_e+a_e)$ by $\nu_e$, yielding errors that can, with the help of the estimate $|H_e-\tilde H_e|_{L^2(\Sigma_t)}\leq c|\Acirc|_{L^2(\Sigma_t)}$, be estimated by
	$$
	R_e^{-2}|\Acirc_e|_{L^2(\Sigma_t)}(|\Acirc_e|_{L^2(\Sigma_t)}+\tau_e) \leq R_e^{-1}|\Acirc|^2_{L^2(\Sigma_t)}+\eta R_e^{-3}+\tau_e^2 R_e^{-3}.
	$$
	  The final equality follows from the translation invariance of the Euclidean area and the translation invariance of the Euclidean volume, which imply, respectively, that
	  $$
	  \int_{\Sigma_t} H_e b_g\cdot \nu_e d\mu_e=\int_{\Sigma_t}b_g\cdot \nu_e d\mu_e=0.
	  $$
	  The term $\Gamma$ can then also be expressed in terms of the approximating sphere yielding as before error terms than can be estimated as in (\ref{trans2}).
	 On the other hand, we can also use the translation invariance of the Euclidean area to conclude
	\begin{align}
	\lambda \int_{\Sigma_t} b_g\cdot \nu_SH_Sd\mu_S&=\int_{\Sigma_t}b_g\cdot \nu_e H_ed\mu_e-2m\lambda\int_{\Sigma_t}\nu_e r^{-2}\phi^{-1}\partial_r\cdot \nu_ed\mu_e \notag
	\\&=-2m\lambda \int_{S_t} \tilde\nu_e \tilde r^{-2}(1-\frac{m}{2\tilde r}) \tilde \partial_r\cdot \tilde\nu_e   d\tilde\mu_e
	+\mathcal{O}(\lambda (|\Acirc|_{L^2(\Sigma_t)}+ R_e^{-2})). \label{trans4}
	\end{align}
	In the second line, we expanded $\phi^{-1}$ up to order $r^{-1}$ and used Lemma \ref{muldelel}. We then estimate using Lemma \ref{finalest}
	\begin{align}
	\lambda (|\Acirc|_{L^2(\Sigma_t)}+ R_e^{-2}) \leq \lambda^2 R_e^2 +|\Acirc|_{L^2(\Sigma_t)}^2R_e^{-2}+ R_e^{-4} \leq cR_e^{-4}-cR_e^2\partial_t\mathcal{W}. \label{trans5}
	\end{align}
	Combining (\ref{trans0})-(\ref{trans5}), we have shown that
	\begin{align}
	\partial_t \tau_g =&\notag \frac{(3-2mR_g^{-1})}{|\Sigma_t|R_g}\bigg(-\frac{d}{ds}\bigg|_{s=0}\int_{S_{R_e}(a_e+sb_g)} \tilde H_S ^2 d\tilde \mu_S-2m\lambda \int_{S_{R_e}(a_e)} \tilde\nu_e\cdot b_g \tilde r^{-2}(1-\frac{m}{2\tilde r}) \tilde \partial_r\cdot \tilde\nu_e   d\tilde\mu_e 
	\\&+\mathcal{O}((\tau_g^2+R_e^{-1}+\eta)R_e^{-3}-R_e^{3}\partial_t\mathcal{W}(\Sigma_t)) \bigg).
\label{first}
	\end{align}
	Due to the rotational symmetry of $g_S$, it is easy to see that the function  $f(a):=\int_{S_{R_e}{(a)}} H_S^2 d\mu_S$ only depends on $\tau_e$ and is in fact analytic in $\tau_e$. Moreover, there holds $d/ds|_{s=0} \tau_e =b_g\cdot b_e R_e^{-1}$ where $b_e=a_e/|a_e|$. It therefore suffices to compute $f(a_e)$ up to terms of order $R_e^{-2}$. Dropping the tilde notation and writing $S=S_{R_e}(a_e)$, we have
\begin{align*}
f(a_e)=\int_S H_e^2 -4m \int_S 2R_e^{-1} (r^{-2}-\frac{m}{2} r^{-3})\partial_r\cdot \nu_e d\mu_e +4m^2 \int_S r^{-4} (\partial_r\cdot \nu_e)^2d\mu_e+\mathcal{O}(R_e^{-3}).
\end{align*}
	After a rotation we may assume that $a_e=(0,0,z)$ with $z=|a_e|$ and thus $\tau_e=R_e/z$. We choose the parametrization
$$
(\theta,\varphi) \mapsto a_e+R_e(\sin\theta\sin\varphi,\sin\theta\cos\varphi,\cos\theta)
$$
and compute the quantities
$$
\nu_e=(\sin\theta\sin\varphi,\sin\theta\cos\varphi,\cos\theta),  \qquad r^2=R_e^2(1+\tau_e^2)+2R_e^2\tau_e\cos\theta.
$$
Furthermore, one can check that
$$\cos(\theta)=(r^2-R_e^2-\delta^2)/(2R_ez)=(r^2-R_e^2(1+\tau_e^2))/(2R_e^2\tau_e)$$
and consequently
$$\partial_r\cdot\nu_e =r^{-1}(R_e+z\cos\theta)=
(r^2+R_e^2(1-\tau_e^2))/(2rR_e).$$ Next, we have $d\mu_e=R_e^2d\varphi d\theta$ and since there is no $\varphi$ dependence, integration of $\varphi$ solely adds a factor of $2\pi$. Finally, we can substitute $\theta\to r$ where the area element transforms via
$$d\theta/dr=-r/(R_ez\sin\theta)=-r/(R_e^2\tau_e\sin\theta)$$ and the boundary data are mapped to $(R_e-z,R_e+z)=(R_e(1-\tau_e),R_e(1+\tau_e))$ in an orientation reversing way. 
This gives
\begin{align*}
f(a_e)=&16\pi-\frac{8\pi m}{R_e^2 \tau_e} \int_{R_e(1-\tau_e)}^{R_e(1+\tau_e)}(r^{-2}-\frac{m}{2}r^{-3})(r^2+R_e^2(1-\tau_e^2))dr
\\&+\frac{2\pi m^2}{R_e^2 \tau_e} \int_{R_e(1-\tau_e)}^{R_e(1+\tau_e)}r^{-5}(r^2+R_e^2(1-\tau_e^2))^2dr+\mathcal{O}(R_e^{-3})
\\&=16\pi -32\frac{\pi m}{R_e}+\frac{6\pi m^2}{R_e^2\tau_e}\log{\frac{1+\tau_e}{1-\tau_e}}+\frac{4\pi m^2}{R_e^2}\frac{5-3\tau_e^2}{(1-\tau_e^2)^2}+\mathcal{O}(R_e^{-3})
\\&=16\pi -32\frac{\pi m}{R_e}+32\frac{\pi m^2}{R_e^2}+\frac{32\pi m^2 \tau_e^2}{R_e^{2}}+\mathcal{O}(\tau_e^4R_e^{-2}+R_e^{-3}),
\end{align*}
where the last equality follows from Taylor's theorem, provided $\delta$ is chosen sufficiently small. Hence from the analyticity it follows that
\begin{align}
\frac{d}{ds} f(a_e)\bigg|_{s=0}= 64 \pi m^2 R_e^{-3}\tau_e b_g\cdot b_e+\mathcal{O}(\tau_e^2R_e^{-3}+R_e^{-4}) = 64 \pi m^2 R_e^{-3}\tau_e +\mathcal{O}(\tau_e^2R_e^{-3}+R_e^{-4}),
\label{second}
\end{align}
where we used that $b_g\cdot b_e=1+\mathcal{O}(R_e^{-1})$. On the other hand, it is easy to see that for any vector $b_e^\perp$ perpendicular to $b_e$
$$
\int_S \nu_e\cdot b_e^\perp r^{-2}(1-\frac{m}{2r})\partial_r\cdot \nu_e d\mu_e =0.
$$
In a similar fashion as above, we thus find
\begin{align}
\notag&-2m\lambda \int_S \nu_e\cdot b_g r^{-2}(1-m/2r^{-1})\partial_r\cdot \nu_e d\mu_e 
\\=\notag&-\frac{\pi m\lambda}{R^3\tau_e^3} b_e\cdot b_g \int_{R_e(1-\tau_e)}^{R_e(1+\tau_e)} (r^2-R_e^2(1-\tau_e^2)) r^{-2}(1-m/2 r^{-1})(r^2+R_e^2(1-\tau_e^2))dr
\\=\notag& \frac{16\pi m\lambda }{3}\tau_e +\mathcal{O}((\tau_e^2+R_e^{-1}) \lambda )
\\=& \frac{32\pi m^2\tau_e}{3R_e^3}+\mathcal{O}(\tau_e^2 R_e^{-3}+cR_e^{-4}-R_e^{3}\partial_t\mathcal{W}(\Sigma_t)),
\label{third}
\end{align}
where we used Lemma \ref{finalest} in the last inequality. Combining (\ref{first})-(\ref{third}) and replacing  $\tau_e$ by $\tau_g$  the claim follows as $3(64-32/3)= 160$.
\end{proof}
We are now ready to prove the main result:
\begin{thm}
	Let $(M,g)$ be $C^3-$close to Schwarzschild with decay coefficient $\eta>0$ and mass $m>0$ and let $\Sigma$ be an embedded sphere. There exist constants $\eta_0(m), \epsilon(m,\eta_0), \delta(m,\eta_0)>0$ and $R_0(m,\eta_0)>0$ such that if $\eta\leq\eta_0$, $r_{\min}\geq R_0$, $\tau_e <\delta/2$ and $|\Acirc|^2_{L^2(\Sigma)}<\epsilon$, then the area preserving Willmore  flow starting at $\Sigma$ exists for all times and converges smoothly to one of the leaves in the foliation $\{\Sigma_\lambda\}$. 
	\label{main result}
\end{thm}
\begin{proof} First we choose $\delta, \epsilon, R_0^{-1}$ small enough such that every admissible surface satisfies the constraints of the previous lemmas. According to Lemma \ref{excess estimate}, we  can choose suitable initial conditions such that the area-preserving Willmore flow can only cease to exist if $\tau_e$ reaches $\delta$. We assume that $\tau_e(\Sigma)=\delta/2$. We would then like to show that $\tau_e\leq \delta$ for all times. For this to hold, it is enough to show that $\tau_g \leq 9\delta /10 $ for all times if $R_0$ is chosen sufficiently large. We suppose for contradiction that there is a first time $T^*>0$ such that $\tau_g(T^*)=9\delta/10$. We may assume that $\tau_g(0) = 6\delta/10 $ and that $\tau_g(t)>6\delta/10$ for all $t>0$. Additionally, we require  $c(\delta^2+\epsilon+\eta) \leq 80m^2\delta$  as well as $c\epsilon<\delta/10$.	 The previous lemma implies that
	$$
		\partial_t \tau_g \leq -160 m^2R_g^{-6}\tau_g+80m^2\delta R_g^{-6}-c\partial_t\mathcal{W}
	$$ 
	for any time $t\in [0,T^*]$.  Hence, by integration and the excess estimate Lemma \ref{excess estimate} we infer that
	\begin{align}3\delta /10  = \tau_g(T^*)-\tau_g(0)\leq \delta/10  +80m^2R_g^{-6}\int_0^{T^*}(-2 \tau_g+\delta)<\delta/10, \label{contradiction} \end{align} 
	 which is of course a contradiction. Hence, no such time $T^*$ exists and Theorem \ref{jachan existence} gives long time existence and convergence of a subsequence to a surface $\Sigma^*$ of Willmore type satisfying (\ref{hypothesis a}). Applying   Lemma \ref{finalest} to the stationary surface $\Sigma^*$, it follows  that $\Sigma^*$ must be strictly mean convex. We may then further decrease $\delta$ such that the uniqueness statement of Theorem \ref{lmstheorem} can be applied and it follows that $\Sigma^*$ is part of the foliation constructed in \cite{lamm2011foliations}. From this, full convergence follows. 
	\end{proof}
Applying the previous result to a small $W^{2,2}$-perturbation of one of the leaves $\Sigma_\lambda$ we obtain the following:
\begin{coro}
	The leaves $\Sigma_\lambda$ are strict local area-preserving maxima of the Hawking mass with respect to the $W^{2,2}$-topology. More precisely, there exists $\tau_0,\Lambda>0$ depending only on $m,\eta$ such that for any $\lambda<\Lambda$ and any surface $\Sigma$ enclosing $\Sigma_\Lambda$ and satisfying $\tau_e\leq \delta$ as well as $|\Sigma|=|\Sigma_\lambda|$ there holds $m_H(\Sigma)\leq m_H(\Sigma_\lambda)$. Equality holds if and only if $\Sigma=\Sigma_\lambda$. Moreover, if $\operatorname{Sc}\geq 0$, then there holds $m_H(\Sigma)\leq m$ with equality if and only if $\Sigma=\Sigma_\lambda$ and the non-compact component of $M\setminus \Sigma_\lambda$ is isometric to the Schwarzschild manifold with mass $m$ and a solid, centred ball removed. In particular, $\Sigma_\lambda$ must be a centred sphere. 
	\label{main coro}
\end{coro}
\begin{proof}
Fix $\epsilon>0$ and $\Lambda>0$ sufficiently small such that the previous theorem can be applied with $R_0:=r_{\min}(\Sigma_\Lambda)$. Let $\delta=\delta(\epsilon,R_0,\eta,m)>0$ be the constant from the previous lemma and $\Sigma$ be a surface enclosing $\Sigma_\Lambda$. As $\Sigma_\Lambda$ is strictly mean convex and star-shaped (see \cite{lamm2011foliations}) it follows that $|\Sigma|=|\Sigma_\lambda|$ for some $\lambda<\Lambda$. If $|\Acirc|^2_{L^2(\Sigma)}\geq\epsilon$  it follows from the integrated Gauss equation (\ref{gausseqn}) that after possibly reducing $\Lambda$ (and thus increasing $R_0$) there holds $m_H(\Sigma)<0$. In the other case, we can apply the previous lemma and note that the area preserving Willmore flow increases the Hawking mass unless $\Sigma$ is a surface of Willmore-type to deduce that $m_H(\Sigma)\leq m_H(\Sigma_\lambda)$ with equality if and only if $\Sigma=\Sigma_\lambda$.   If the scalar curvature is non-negative, we can use the inverse mean curvature  flow (see \cite{huisken2001inverse}) starting at $\Sigma_\lambda$ to show that $m_H(\Sigma_\lambda)\leq m$. If equality holds, the rigidity statement in \cite{huisken2001inverse} readily implies that $\Sigma_\lambda$ must be a centred sphere in the Schwarzschild manifold. 
\end{proof}
\begin{rema}
It would of course be desirable to know if, and if so, in what sense, the leaves $\Sigma_\lambda$ are global maximizers of the Hawking mass. Even in the exact Schwarzschild space, this seems to be a difficult problem: Connecting several spheres close to the horizon by small catenoidal bridges, one can construct centred spherical surfaces with arbitrarily large Hawking mass which are additionally homologous to the horizon. Hence, one cannot expect any maximizing property without excluding a large compact set. On the other hand, a straight forward computation reveals that the Hawking mass of the sphere $S_R(R^\beta e_3)$ tends to $m$ as $R\to\infty$ for any $1/2<\beta<1$. Such surfaces eventually avoid any compact set and become totally off-centred. 
\end{rema}

\appendix
\section{Outline of the proof of Theorem \ref{jachan existence}}
For convenience we give an outline of Jachan's proof of Theorem \ref{jachan existence}. \\
\textbf{Notation and setting.} Let $T$ be a general tensor. We will use the star notation as well as the polynomial expressions
$$
P^k_j(T):=\sum\limits_{i_1+\dots+i_j=k}\nabla^{(i_1)}T*\dots*\nabla^{(i_k)}T, \qquad
P^{k,s}_j(T):={\sum\limits_{i_1+\dots+i_j= k, i_l\leq s}}\nabla^{(i_1)}T*\dots*\nabla^{(i_k)}T, 
$$
for $0\leq s\leq k$ and $l\in\{1,\dots,j\}$. We denote expressions of the form $T(P(\nu,\dots,\nu,\cdot,\dots,\cdot))^T$ by $T^{(\nu)}$. Here, $P$ is any permutation of the arguments. Examples of this notation are given by the Gauss equation
$\operatorname{Rm}^\Sigma=\operatorname{Rm}*1+P_2^0(A)$ and by Simon's identity $\Delta A=\nabla H+P^1_1 (\operatorname{Rc}^{(\nu)})+P^{0}_3(A)+A*\operatorname{Rc}$. The assumptions of Theorem \ref{jachan existence} and Simon's diameter estimate, Lemma 1 in \cite{simon1993existence}, imply that the flow stays in a compact region of the manifold where henceforth the ambient curvature and all its derivatives are uniformly bounded. For ease of notation, we will assume that all derivatives can be bounded by the same constant. \\
\textbf{Evolution equations.}
 As the flow is a fourth order parabolic equation, short time existence follows and we let $T>0$ be the maximal existence time of the flow. Following \cite{kuwert2002gradient}, in order to understand the properties of the flow, higher order estimates for the curvature are needed. This is done be studying the evolution of the second fundamental form. Using a rough version of Lemma \ref{der swap}, the Simon's identity, the Gauss equation, the standard formula for the change of the second fundamental form under a normal variation as well as the flow equation one obtains (c.f. Lemma 2.3 in \cite{kuwert2002gradient})
 \begin{lem}
 	Let $\Sigma_t$ be an area preserving Willmore flow. There holds
 	\begin{align*}
 	(\partial_t+\Delta^2)A=&P^2_3(A)+P^0_5(A)+\sum_{k+l=2} \nabla^l A * \nabla ^k \operatorname{Rc}^{(\nu)}+1*\nabla^3 \operatorname{Rc}^{(\nu)}+A*P^0_2\operatorname{Rc}^{(\nu)}+P^0_3 A*\operatorname{Rc}^{(\nu)}
 	\\&+\lambda*(\nabla^2 A+P^0_3(A)+A*\operatorname{Rc}^{(\nu)}).
 	\end{align*}
 \end{lem}
Using induction and Lemma \ref{der swap} again one also obtains higher order versions of this equation, see Proposition 2.4 in \cite{kuwert2002gradient}. 
\\ \textbf{Integral curvature estimates.}
The idea is now to get localized integral curvature estimates from the evolution equation. To this end, we choose a cut off function $\eta$ in $M$ which becomes a function on $\Sigma\times[0,T)$ after restriction. We then adapt Lemma 3.2 from \cite{kuwert2002gradient} to our setting where we have to keep track of additional terms arising from the ambient curvature.  This gives an estimate for the evolution of an integral of the form $\int_\Sigma \eta^s |T|^2 d\mu$ for a general tensor $T$. Inserting $T=\nabla^m A$
and estimating $|\nabla \eta|\leq \Lambda_1$ as well as $|\nabla^2 \eta| \leq \Lambda_2+|A|\Lambda_1$ gives the following crude estimate 
\begin{lem}
	Let $m\geq0$ and $s\geq 2m+4$. There exists a constant $c$ which depends on $s,m$ such that
	\begin{align*}
	&\partial_t \int_{\Sigma_t} \eta^s |\nabla^m A|^2 d\mu +\frac{3}{4}\int_{\Sigma_t} \eta^s |\nabla^{m+2}A|^2d\mu+\int_{\Sigma_t} \eta^{s-2}|\nabla \eta|^2 |\nabla^{m+1} A|^2 d\mu\\ 
	\leq& c\Lambda_1^2\int_{\Sigma_t} \eta^{s-2} |\nabla^{m+1} A|^2 d\mu + c(1+\Lambda_1^4+\Lambda_2^2)\int_{\Sigma_t} \eta^{s-4} |\nabla^m A|^2 d\mu +\lambda ^2\int_{\Sigma_t} \eta^s|\nabla^m A|^2 d\mu \\
	&+c\int_{\Sigma_t} \eta^s (P^{2m,m}_6(A)+\nabla^m A*[P^{m+2}_3(A)+\nabla^{m+3}\operatorname{Rc}^{(\nu)}+\sum_{k+l=m+2}\nabla^k A*\nabla^l \operatorname{Rc}^{(\nu)}])d\mu 
	\\&+\int_{\Sigma_t} \eta^s\nabla^mA*\sum_{k+l=m}[P^k_3(A)*\nabla^l\operatorname{Rc}^{\nu}+\nabla^k(A)*P^l_2\operatorname{Rc}^{(\nu)}]d\mu
	\\&+\int_{\Sigma_t}\eta^s*\lambda*\nabla^m(A)*\sum_{k+l=m} \nabla^kA*\nabla^l \operatorname{Rc}^{(\nu)}d\mu.
	\end{align*}
\end{lem}
This should be compared to Proposition 3.5 in \cite{kuwert2002gradient}. The next step is then to rewrite the terms $\nabla^k \operatorname{Rc}^{(\nu)}$ in terms of ambient derivatives of $\operatorname{Rc}$ and derivatives of the second fundamental form. The resulting terms can then be estimated and sometimes absorbed using Young's inequality. Moreover, lower order terms are estimated by their $L^\infty$-norm.  
 This process is quite lengthy but rather straight-forward and results in the following evolution inequality.
\begin{lem}
	Let $m\geq 1$ and $s\geq 2m+4$. Then there exists a constant $c$ depending only on $s,m$ as well as another constant $C$ which also depends on  $\Lambda_i$, the estimates on the ambient curvature and $|\nabla ^k A|_{L^\infty(\Sigma_t\cap\{\eta>0\})}$ for $k\leq m-3$ such that
	\begin{align*}
	&\partial_s\int_{\Sigma_t} \eta^s |\nabla^m A|^2 d\mu +\frac12 \int_{\Sigma_t}\eta^s |\nabla^{m+2}A|^2 d\mu 
	\\\leq& c(\lambda^2+|A|_{L^\infty(\Sigma_t\cap\{\eta>0\})}^4)	\int_{\Sigma_t}\eta^s |\nabla^m A|^2 d\mu+ C|\eta>0|+C(1+|A|_{L^\infty(\Sigma_t\cap\{\eta>0\})}^4)|A|_{L^2(\Sigma_t\cap\{\eta>0\})}^2.
	\end{align*}
\end{lem}
This should be compared to Proposition 4.5 in \cite{kuwert2002gradient}. A similar evolution inequality can also be derived for $m=0$, however in this case, a delicacy appears: The ultimate goal is to interpolate between the highest order term and $|A|^2_{L^2}$. Hence, we cannot make use of $L^\infty$-estimates in the evolution equation for $m=0$. To circumvent this problem, one uses the Michael-Simon-Sobolev inequality and absorbs the higher order terms created by this procedure. This is very similar to the proof of Lemma 4.3 in \cite{kuwert2002gradient}. It is only possible to absorb the higher order term if one assumes a small energy condition
\begin{align} |A|^2_{L^2(\Sigma_t\cap\{\eta>0\})}<\epsilon_0 \label{appendix small curvature}
\end{align} for some constant $\epsilon_0>0$. In fact, one obtains the following lemma.
\begin{lem}
	Suppose (\ref{appendix small curvature}) holds and that $s\geq 4$. Then there exists a constant $c$ depending only on $s$ and a constant $C$ which also depends on the ambient curvature bounds and and $\Lambda_i$ such that
	\begin{align*}
	&\partial_t \int_{\Sigma_t} \eta^s |A|^2 d\mu+\frac12 \int_{\Sigma_t} \eta^s(|\nabla ^2A|^2+|\nabla ^2 A|^2|A|^2+|A|^6) d\mu \\\leq& c|\lambda|^2\int_\Sigma \eta^s |A|^2d\mu+c\int_\Sigma \eta^s|\overline{\nabla} \operatorname{Rc}|^2 d\mu +C|A|^2_{L^2(\Sigma_t\cap\{\eta>0\})}.
	\end{align*}
	\label{appendix int m=0}
\end{lem}
 As can be seen in Proposition 4.4. in \cite{kuwert2002gradient}, the same phenomenon appears in the Euclidean case and is consequently \textit{not} caused by the effects of the ambient curvature.
\\\textbf{A-priori estimates.}
Under the small curvature assumption, the integral curvature estimates can be turned into $L^\infty$-estimates by Gronwall's inequality and by Sobolev-type interpolation estimates in the spirit of Lemma 2.8 in \cite{kuwert2001willmore}. Arguing essentially as in Proposition 4.6 in \cite{kuwert2002gradient} while keeping track of the additional terms arising from the ambient curvature and the Lagrange parameter $\lambda$ one obtains
\begin{lem}
	If $|A|^2_{L^2(\Sigma_t\cap\{\eta>0\})}<\epsilon_0$ holds, then for any $k\in\mathbb{N}$ there exist constants $C,\tilde C$ depending on $k,T,\Lambda_i,\sup_{t\in(0,T)}|\eta>0|$, the estimates for the ambient curvature, $|\lambda|_{L^2((0,T))}$ and $|\nabla^i A|_{L^2(\Sigma_0\cap\{\eta>0\})}$ for $i\leq k$ and $i \leq k+2$, respectively, such that
	$$
	\sup_{t\in[0,T)} |\nabla ^k A|_{L^2(\Sigma_t\cap\{\eta=1\} )}\leq C, \qquad
	\sup_{t\in[0,T)} |\nabla ^k A|_{L^\infty(\Sigma_t\cap\{\eta=1\} )}\leq \tilde C.
	$$ 
	\label{appendix a priori}
\end{lem}
Similarly, interior estimates without the dependence on the initial data can be proven, c.f. Theorem 3.5 in \cite{kuwert2001willmore}.
\begin{lem}
If $|A|^2_{L^2(\Sigma_t\cap\{\eta>0\})}<\epsilon_0$ then for any $k\in\mathbb{N}$ and any $t_0\in(0,T)$ there exists a constant $C$ depending on $k,t_0^{-1},T,\Lambda_i,\sup_{t\in(0,T)}|\eta>0|$, the estimates on the ambient curvature and $|\lambda|_{L^2((0,T))}$ such that
$$
\sup_{t\in[t_0,T)} |\nabla ^k A|_{L^2(\Sigma_t\cap\{\eta=1\} )}\leq C, \qquad
\sup_{t\in[t_0,T)} |\nabla ^k A|_{L^\infty(\Sigma_t\cap\{\eta=1\} )}\leq C.
$$ 
	
\end{lem}
\noindent
\textbf{Curvature concentration at a singularity.}
In the next step, it is shown that a singularity can only occur, if the curvature concentrates around one point, a condition that will be specified later. To this end, one notes that a singularity can only occur if $|\lambda|_{L^2((0,T))}$ or $|A|_{C^k(\Sigma_t)}$ blows up for some $k\in\mathbb{N}$, because otherwise the flow converges smoothly as time approaches $T$ and can then be restarted. Assuming for now that $|\lambda|_{L^2((0,T))}$ is bounded, by Lemma \ref{appendix a priori} one obtains uniform curvature estimates if the small curvature condition is satisfied on every ball of a small radius $\rho>0$. For a small $\rho$, this is certainly the case at $t=0$, one can in fact assume that $\int_{\Sigma_t\cap B_\rho}|A|^2d\mu <\frac12 \epsilon_0$ for each ball of radius $\rho$. Integrating the inequality in Lemma \ref{appendix int m=0} one sees that the curvature can only concentrate after an amount of time $\tilde T$ depending on $\rho$, the ambient curvature bounds and $|\lambda|_{L^2((0,T))}$. The last condition seems somewhat restrictive, however the point is to apply the above to a time close to the singular time. Namely, if $|\lambda|_{L^2((0,T))}$ is bounded, then $|\lambda|_{L^2((t,T))}$ approaches zero as $t\to T$.  From this, one can then deduce that a singularity can only occur if the curvature concentrates in a ball of arbitrarily small radius, as otherwise the flow could be continued past the singular time. More precisely, similarly to Theorem 1.2 in \cite{kuwert2002gradient} we have
\begin{lem}
	Let $\lambda\in L^2((0,T))$  and $T<\infty$. Then there exists $\delta>0$, radii $\rho_i\to 0$ and times $t_i\to T$ such that
	$$
	\delta \leq \lim_{i\to\infty} \sup_{x\in M} \int_{\Sigma_{t_i}\cap B_{\rho_i}(x)} |A|^2 d\mu.
	$$
	\label{appendix curvature concentration}
\end{lem} 
The reasoning so far only relied on the fact that one has ambient curvature bounds and that the Michael-Simon-Sobolev inequality is available. However, in the following, the special geometry of an asymptotically Schwarzschild manifold will play a more important part. \\
\textbf{$L^2-$estimates for $\lambda$.} In the more restrictive setting of this paper, Lemma \ref{good estimates} together with Lemma \ref{excess estimate} easily imply a uniform $L^2-$estimate on $\lambda$. However, one can even show an $L^2-$estimate in the more general setting of Theorem \ref{jachan existence}. Let $f:=g(x,\nu)$, where $x$ is the position vector in the chart at infinity. In the Euclidean space, the scaling properties of the area and the scaling invariance of the Willmore energy imply
$$
2|\Sigma|=\int_{\Sigma} fHd\mu, \qquad 0=\int_{\Sigma} fW d\mu.
$$
Now, the point is that these properties still hold for asymptotically Schwarzschild metrics up to a small error term, as can be verified by direct computation. In fact, provided $R_0$ is sufficiently large, one has
$$
|\Sigma| \leq \int_{\Sigma} fHd\mu, \qquad \bigg|\int_{\Sigma} fW d\mu\bigg|\leq cR_0^{-1}\int_\Sigma H^2 d\mu.
$$
Using these estimates, one can then proceed to show the inequality 
$$
\lambda(t)^2 \leq c\mathcal{W}(\Sigma_t) \int_{\Sigma_t} (W+\lambda H)^2d\mu +cR_0^{-1}\mathcal{W}(\Sigma_t)^2|\Sigma|^{-2}=-c\mathcal{W}(\Sigma_t)\partial_t\mathcal{W}+cR_0^{-1}\mathcal{W}(\Sigma_t)^2|\Sigma|^{-2}.
$$
In the first inequality we could have actually replaced $\lambda$ by any real number, the second one follows from (\ref{trick}). As the Willmore energy can be assumed to be uniformly bounded from above and below, this implies that $\lambda$ is in $L^2((0,T))$.
\\
\textbf{Blow-up analysis and long time behavior.} So far, we have shown that if a singularity develops, a curvature concentration occurs, see Lemma \ref{appendix curvature concentration}. The idea is now to blow up at the point where the curvature concentrates and to derive a contradiction under suitable initial conditions. Thanks to the chart at infinity, the blow up can be performed in a convenient way. To this end, let $\rho_i, t_i$ be as in Lemma \ref{appendix curvature concentration} and $F:\Sigma\times [0,T)$ the area preserving Willmore flow. Following the approach in section 4 of \cite{kuwert2001willmore}, one then defines the parabolic rescaling $F_i:=\rho^{-1}F(x,\rho_i^{4}t+t_i)$ defined on $\Sigma\times[-\rho_i^{-4}t_i,\rho_i^{-4}(T-t_i)]$. Additionally, one rescales the metric in the asymptotic chart via $g_i(p)=g(\rho_i p)$.  It then follows that $F_i$ is an area preserving Willmore flow with respect to the metric $g_i$. Moreover, there holds $\mathcal{W}(F_i(\Sigma,0))=\mathcal{W}(F(\Sigma,t_i))$ where the Willmore energy is understood with respect to the different background metrics. The quantity $\rho_i^{-4}(T-t_i)$ can be uniformly estimated from below by the reasoning which was used to prove Lemma \ref{appendix curvature concentration}. From this it follows that the existence time of the rescaled flows is uniformly bounded from below, in fact all flows exist on an interval $[-1,\xi]$ for some $\xi>0$. It is then easy to see that close to the rescaled flow $F_i$, the ambient curvature terms of the rescaled metric $g_i$ converge to $0$ (here one essentially uses again that the original flow stays in a compact region). In the same way, the $L^2-$estimate for $\lambda$ from the previous subsection implies that the rescaled Lagrange parameter satisfies $|\lambda_i|_{L^2([-1,\xi])}\to 0$. The rest of the argument is then essentially the same as in \cite{kuwert2001willmore}: one obtains uniform gradient estimates and shows that the rescaled flow converges to a stationary Willmore immersion in $\mathbb{R}^3$. While the $L^2-$estimates for $\lambda$ required some additional work, a lower bound on the area of the flow is automatic (contrary to the normal Willmore flow, see Theorem 5.2 in \cite{kuwert2001willmore}) which implies that the limiting surface is non-compact. Moreover, the curvature around the point of concentration is not lost in the limit which implies that the limiting surface is not a plane. We therefore obtain
\begin{lem} If a curvature concentration occurs, the surface $\rho_i^{-1}F(\Sigma,t_j)$ converges locally uniformly to a non-compact Willmore immersion $\hat \Sigma$ in $\mathbb{R}^3$ which is not a plane.
\end{lem}  
The same reasoning can also be applied if the curvature concentration occurs at $T=\infty$. Now, if one additionally requires that $\mathcal{W}(\Sigma)\leq 8\pi -\rho$ holds for some $\rho>0$, then for $R_0$ sufficiently large, this inequality also holds for the Euclidean Willmore energy. Moreover, the inequality is also true for the Euclidean Willmore energy of $\hat\Sigma$. Then, by a spherical inversion, one obtains an embedded compact Euclidean Willmore surface with a point singularity and Euclidean Willmore energy less than $8\pi$. According to Lemma 5.1 in \cite{kuwert2004removability}, the point singularity can be removed. But then by the classification of genus $0$ Willmore surfaces by Bryant, see \cite{bryant1984duality}, the inversion of $\hat\Sigma$ has to be a round sphere. This, however, implies that $\hat\Sigma$ is a plane, a contradiction. From this we deduce the following result.
\begin{lem}
	Let $\mathcal{W}(\Sigma)<16\pi-\rho$ and $R_0$ be sufficiently large. Then an area preserving Willmore flow enclosing the Ball $B_{R_0}(0)$ for all times cannot develop a singularity at any finite or infinite time.
\end{lem}
Finally, the previous lemma gives uniform curvature estimates for all times and then a standard compactness argument implies the subsequential convergence to a surface of Willmore type.

\end{document}